\documentclass[12pt]{amsart}
\usepackage{inputenc}
\usepackage{amsmath}
\usepackage{amssymb}
\usepackage{amsthm}
\usepackage[all]{xy}
\usepackage{graphicx}
\usepackage{xcolor}
\usepackage{subcaption}
\usepackage{pstricks}
\usepackage{pst-plot}
\usepackage{cancel}
\usepackage{multirow}
\usepackage{siunitx}
\usepackage{xparse}
\usepackage{float}
\usepackage{comment}
\usepackage{blindtext}
\usepackage{geometry}
\usepackage{tikz-cd}
\usepackage{xfrac}
\usepackage{xparse}
\usepackage{faktor}
\usepackage{easytable}

\newcommand{\CP}{\mathbb{CP}}
\newcommand{\C}{\mathbb{C}}
\newcommand{\R}{\mathbb{R}}
\newcommand{\T}{\mathbb{T}}
\newcommand{\Z}{\mathbb{Z}}
\newcommand{\bD}{\mathbb{D}}
\newcommand{\cX}{\mathcal{X}}

\DeclareMathOperator{\std}{std}
\newcommand{\om}{\omega}
\newcommand{\Om}{\Omega}

\DeclareMathOperator{\Symp}{Symp}
\DeclareMathOperator{\Diff}{Diff}

\DeclareMathOperator{\Hom}{Hom}

\DeclareMathOperator{\id}{id}
\DeclareMathOperator{\PD}{PD}

\newtheorem{thm}{Theorem}[section]
\newtheorem{prop}[thm]{Proposition}
\newtheorem{exmp}[thm]{Example}
\newtheorem{cor}[thm]{Corollary}
\newtheorem{rmk}[thm]{Remark}
\newtheorem{defn}[thm]{Definition}
\newtheorem{lemma}[thm]{Lemma}

\newtheorem{ques}{Question}
\newtheorem{pblm}{Problem}

\title{Almost toric fibrations on K3 surfaces via degenerations}
\author{Pranav Chakravarthy and Yoel Groman}
\begin{document}

\begin{abstract}
For K\"ahler K3 surfaces we consider Kulikov models of type III  tamed by a symplectic form. Our main result shows that the generic smooth fiber admits an almost toric fibration over the intersection complex, which inherits a natural nodal integral affine structure from almost toric fibrations of the boundary divisors. We prove that a smooth anti-canonical hypersurface in a smooth toric Fano threefold, equipped with a toric Kähler form, admits a symplectic Kulikov model. Moreover, we demonstrate that the induced integral affine structure on the intersection complex is integral affine isomorphic (up to nodal slides) to the nodal integral affine structure considered by Gross and Siebert on the boundary of the moment polytope. 
\end{abstract} 
\maketitle
\tableofcontents

\section{Introduction}
\subsection{Statement of the main result}
A \emph{K3 surface} is a closed complex analytic surface which is simply connected and has trivial canonical bundle. Consider a  family  $$\pi: \cX \rightarrow D^*_\epsilon$$ over a punctured open disk $D^*_\epsilon:=\{z \in \C~|~ 0<|z| < \epsilon\}$ whose fiber over any point is a smooth K3 surface. A Theorem by Kulikov-Persson-Pinkham \cite{Kulikov,PerssonPinkham} states that $\pi$ can be extended to $D_\epsilon$  as a \emph{Kulikov model}. 

\begin{defn}(Kulikov model)
We call a family of K3 surfaces $$\pi: \cX \rightarrow D$$ over an open disk $D$ a \emph{Kulikov model} if the total space $\cX$ is smooth, has trivial canonical bundle, and the central fibre $\cX_0 = \bigcup_i X_i$ is  a reduced normal crossings divisor. The central fibre $\cX_0$ is called the \emph{Kulikov surface.}
\end{defn}

To the central fiber $\cX_0$ is associated the \emph{dual intersection complex}. This is a simplicial complex in which the $k$-simplices are $(k+1)$-fold intersections $X_I$ of the components of $\cX_0$ and the boundary faces of the simplex $X_I$ are the simplices $X_J$ with $J\subset I$ and $|J|=|I|-1$.

We say that the degeneration is \emph{maximal} if the monodromy around the origin of the middle homology of the fiber is maximally unipotent. In this case,  the  Kulikov model is of \emph{type III}. This means the dual intersection complex is homeomorphic to $S^2$. Moreover, each component $X_i$ of $\cX_0$ is a rational surface and the double curves  $X_{ij}:=X_i\cap X_j$ occurring on a given component form an anti-canonical cycle. For a reference see \cite{Huybrechts}. 

To study the symplectic topology of a K3 surface via its degeneration we introduce the following notion.
\begin{defn}\label{Defn:AdmissibleKulikovModel}
    A \emph{symplectic Kulikov model} $\pi:\cX \rightarrow D$ of Type III, is  a Kulikov model of type III endowed with 
    \begin{itemize}
        \item a symplectic form $\omega$ taming the complex structure on $\cX$ and such that at each point of triple intersection $X_{ijk}$ the 3 submanifolds $X_{ij}$, $X_{jk}$ and $X_{ki}$ are symplectically orthogonal, and $\omega$ is compatible with the complex structure in an open neighborhood of each point of $X_{ijk}$ in $\cX$, and, 
        \item for each $i,$ an almost toric fibration (ATF) $\mu_i:(X_i,\om)\to B_i$ whose almost toric boundary consists of the cycle $\sum_j X_{ij}.$
    \end{itemize}
\end{defn}

\begin{rmk}
According to a recent result by \cite{Li-Min-Ning} such almost toric fibrations exist for symplectic log Calabi-Yau pairs $(X_i,\sum_j X_{ij})$. 
\end{rmk}

A symplectic manifold $(X,\om)$ is said to admit a symplectic Kulikov model of Type III if $(X,\om)$ is symplectomorphic to the fibre of a degeneration $\pi: \cX \to D$ satisfying the properties of Definition~\ref{Defn:AdmissibleKulikovModel}. 
\begin{rmk}
    The compatibility assumption near the vertices is used in \S\ref{SubsecTransportContinuity} to establish continuity of the symplectic parallel transport up to the central fiber for points flowing into the triple intersections. For points flowing into lower strata of the central fiber we do not know how to prove such continuity when the symplectic form is merely tame. Since we do not require compatibility everywhere, we have to do without continuity everywhere of parallel transport. This somewhat complicates the argument. 
    Our reason for allowing symplectic forms that merely tame the total space is that the forms we construct in \S\ref{Sec:KulikovfromFano} for small resolutions are not compatible with $J$ on certain regions away from the triple intersections. See Remark \ref{Remark:Modifiedform_SymplecticOrthogonal} for further discussion. 
\end{rmk}

With a type III degeneration of K3 surfaces, we can also associate a cell complex, the \emph{intersection complex} $\Delta_{\cX}$, as follows. We have a $0$-cell for each triple intersection. We connect two distinct $0$-cells by a $1$-cell if there is a double curve containing them.  By the previous paragraph the edges formed in this way are in bijection with the double curves. Finally, we attach a $2$-cell to any cycle in this graph whose corresponding cycle of curves occurs as the boundary divisor of a component of $\cX_0$. Note that $\Delta_{\cX}$ is dual to the dual intersection complex $\Delta^*_{\cX}$ of the Kulikov surface $\cX_0$. It is therefore also homeomorphic to a $2$-sphere.

To proceed, we introduce an integral affine structure with nodal singularities on the intersection complex $\Delta_{\cX}$.

\begin{lemma}\label{lmIntAffStr}
    Given a symplectic Kulikov model $\cX$ of type III there is a nodal integral affine structure on the intersection complex $\Delta_{\cX}$ which is smooth in a neighborhood of the $1$-skeleton and whose restriction to the 2-cell in $\Delta_{\cX}$ associated with the component $X_i$ is integral affine isomorphic to $B_i$. 
\end{lemma}

This is proven below as Theorem \ref{tmNAffine}. Our aim is to prove the following theorem.

\begin{thm}\label{Thm:LTFongoodKulikov}
The smooth fibre $\cX_t:=\pi^{-1}(t),t\neq 0$, endowed with the symplectic form $\omega|_{\cX_t}$ admits an almost toric fibration $\mu: \cX_t\to \Delta_{\cX}$ inducing the nodal integral affine structure of Lemma \ref{lmIntAffStr}.
\end{thm}

\begin{rmk}
    The dual intersection complex, $\Delta^*$ also carries a triangulated integral affine structure with singularities. See, e.g., \cite{Engel2018}. 
    While the cells in $\Delta_{\cX}$ are dual to those in $\Delta^*$, it is not entirely clear how to relate the integral affine structures. For example, applying a discrete Legendre transform to $\Delta^*$  produces an IAS on $\Delta_{\cX}$ with all the nodes of a given cell occurring at one point. However, this holds, up to nodal slides, if and only if 
   the restriction of $\omega$ to each component of the central fiber is exact on the interior.   
\end{rmk}

\subsection{Hypersurfaces in smooth toric Fano}
We discuss the case of the degeneration of an anti-canonical hypersurface in a smooth toric Fano threefold $X$. Namely, we let $s_0,s_1$ be sections of the anti-canonical bundle so that $s_0^{-1}(0)$ is the sum of the toric divisors and $s_1$ is generic. Consider the family $t\mapsto V(s_0+ts_1)\subset X$. The total space $E$ of this degeneration is singular at double points occurring in the $1$-dimensional toric stratum and thus doesn't fit into the framework of Definition \ref{Defn:AdmissibleKulikovModel}. In the complex category, one can carry out a small blow up at each of the double points to obtain a Kulikov model of type III. This blow-up procedure is not unique. Indeed, for any of the $24$ double points we have a binary choice between two toric divisors in which to carry out the small blow up. 

In the symplectic category, the possibility of small blow up is much more constrained. By the results of \cite{Conifold}, only for a small number of the possible small resolutions does there exist a symplectic form taming the complex structure on the total space of the degeneration. Moreover, if we start out with a symplectic form $\omega$ on the total space, the symplectic crepant resolution involves a global modification of the symplectic form which affects the generic fiber $V$. This modification can be made arbitrarily small, but not $0$. Thus, if we wish to fix the symplectic form $\omega$ on a smooth fiber $V$, it is a delicate question whether or not $(V,\omega)$ admits a symplectic Kulikov model.

We turn to formulate a result in this direction. Given a toric K\"ahler form $\omega$ on the ambient space $X$, let $P_{\omega}$ be the associated moment polytope. On $\partial P_{\omega}$ there is a nodal integral affine structure $\mathcal{A}_{GS}$ considered in the Gross-Siebert program. It extends the standard piecewise affine structure on $\partial P_{\omega}$ with nodal singularities occurring in the interiors of the $1$-skeleton. See \S\ref{SecGrossSiebert}  for details. 

\begin{thm}\label{thmToricKulikov}
  Let $V$ be a smooth anti-canonical hypersurface in a smooth toric Fano 3-fold $X$. Then the restriction of any toric K\"ahler form admits a symplectic Kulikov model of Type III. The integral affine structure induced by Lemma \ref{lmIntAffStr} on the associated intersection complex is integral affine isomorphic up to nodal slides with $(\partial P_{\omega},\mathcal{A}_{GS}).$
\end{thm}

\begin{cor}\label{corToricKulikov}
     For any ambient toric form $\omega$, a smooth anticanonical hypersurface $V$, equipped with the restriction of $\omega$, admits an almost toric fibration over $(\partial P_{\omega},\mathcal{A}_{GS}).$
\end{cor}

\begin{rmk}
     We expect this corollary to hold for a smooth anti-canonical hypersurface in a possible non-smooth toric Fano 3-folds.
\end{rmk}
While the statement of Theorem \ref{thmToricKulikov} appears expected at first glance, at closer inspection it is more surprising. As already noted, a symplectic Kulikov model is obtained by a careful choice of small resolution. The possible choices we consider depend on an arbitrary choice of ordering of the boundary divisors. These choices affect the ATF's on the boundary components. The final effect of this on the integral affine structure is far from immediately obvious. The fact that after these resolutions the integral affine structure remains isomorphic (up to nodal slides) to the Gross-Siebert structure $\mathcal{A}_{GS}$, regardless of the choice, is quite noteworthy. See the next subsection for further discussion of this point.

\subsection{The general context}
According to the SYZ conjecture \cite{SYZ}, a maximal degeneration of Calabi-Yau manifolds should give rise to a special Lagrangian fibration on smooth fibers near the large complex structure limit. Such a special Lagrangian fibration gives rise to a pair of integral affine structures on the regular locus of its base $B$, which are related to each other by a Legendre transform \cite{Gross2012}. These affine structures arise by considering the fluxes of $\Omega$ and $\omega$ respectively on the torus fibers. According to a conjecture by \cite{KontsevichSoibelman01}, $B$ with one of these IAS should be interpreted as the essential skeleton, a certain subcomplex in the dual intersection complex which in our case coincides with it. Moreover, this IAS can be constructed, at least in the projective case, via non-Archimedean geometry \cite{NASYZ}. The other one, its Legendre dual, which comes from the Arnold Liouville coordinates, can be interpreted in our case as living on the intersection complex $\Delta_{\cX}$. For a recent survey, see \cite{Li2022survey}.

For the purpose of studying the Fukaya category it suffices to consider instead a \emph{weak SYZ fibration} on the general fiber. Let us give a tentative definition of this notion. A weak SYZ fibrations is a continuous involutive system over a base $B$ such that the singular locus is a codimension $2$ CW complex and such that the fiber over the regular locus is a Maslov $0$ Lagrangian torus. The full definition needs to be supplemented by some convexity properties near the singularities, but we will not explore this here. Such a weak SYZ fibration would go a long way towards studying the Fukaya category of the general fiber.

This motivates the following general problem. 
\begin{pblm}\label{pblm1}
    Given a Kahler manifold $V$ with a maximal degeneration $\cX$ whose total space is tamed by a symplectic form $\omega$ produce a weak SYZ fibration on $V$ whose base is dual in an appropriate sense to the essential skeleton. 
\end{pblm}

Theorem \ref{Thm:LTFongoodKulikov} addresses Problem \ref{pblm1} in dimension $2n=4$. Note the fibers of an $ATF$ over a K3 are automatically Maslov $0$. More precisely, since a K3 is simply connected, up to homotopy there is a unique smooth trivialization of the canonical bundle. On the other hand, for any nodal fibration there is a unique such trivizalization for which the regular fibers are Maslov $0$. See, e.g., \cite[\S7.6]{GromanVarolgunes21}. The construction in Theorem \ref{Thm:LTFongoodKulikov} involves the additional input of ATF's on the components of the central fiber. The log Calabi-Yau pair $(X,D)$ does not uniquely determine the ATF on $X$ with boundary divisor $D$. It is an interesting question whether one can obtain genuinely different ATF's on K3 from different such choices. This would lead, e.g., via the construction of \cite{KontsevichSoibelman06}, to multiple mirrors.

A variant of Problem \ref{pblm1} arises in the \emph{Gross-Siebert program} \cite{GrossSiebertI,GrossSiebertII}. The latter considers \emph{toric degenerations} instead of maximal degenerations. These have a central fiber consisting of a configuration of toric varieties. Unlike the maximal degeneration case, the total space of a toric degeneration is singular. Given a toric degeneration one can define its intersection complex $\Delta_{\cX}$ \cite{GrossSiebertI}. The central fiber $\cX_0$ is endowed with a continuous map $\mu:\cX_0\to \Delta_{\cX}$ which restricts to the toric moment map. The datum of the degeneration induces an extension of the piecewise affine structure to a strict affine structure on the complement of a codimension 2 set contained in the lower dimensional strata. In the two dimensional case this is a nodal integral affine manifold. However, note the difference to the integral affine structure of Lemma \ref{lmIntAffStr} for which the nodes all occur in the interiors of the top cells whereas in the Gross-Siebert model they occur on the baricenters of the edges in the 1-skeleton.

\begin{pblm}\label{pblm2}
    Given a Kahler manifold $V$ with a toric degeneration $\cX$ whose total space is compatible with a closed 2-form $\omega$ which is non-degenerate on the regular locus and toric on the central fiber produce a weak SYZ fibration on $V$ whose base is the intersection complex. 
\end{pblm}

Our Theorem \ref{thmToricKulikov} addresses Problem \ref{pblm2} in the particular case of anti-canonical hypersurfaces in smooth toric fano in dimension $2n=4$. As already mentioned, we reduce this case to the previous one via small resolution. Nevertheless, the end result is an ATF whose base is the one appearing in the Gross-Siebert program. This leads to the following general question:
\begin{ques}
    Does \ref{thmToricKulikov} apply more generally to toric degenerations of 
    K\"ahler K3 surfaces? That is, given a toric degeneration $\cX\to D^*_\epsilon$ of $X$ is there a symplectic small resolution of $\cX$ so that induced integral affine structure on $\Delta_{\cX}$ coincides with the Gross-Siebert one?
\end{ques}

\subsection{Outline of the proof of Theorem \ref{Thm:LTFongoodKulikov}}
The construction is based on pulling back on ATF defined on the central fiber via parallel transport.  This idea goes back to \cite{Ruan:2001} who applied this to toric degnerations of Calabi-Yau hypersurfaces. The main problem for us is, \emph{how does one extend this torus fibration smoothly over the $1$-skeleton?} We have to deal now with the points that flow into the normal crossings singularity.

Our construction proceeds in two steps:
\begin{itemize}
    \item a \emph{preparatory step} which uses soft symplectic topology. As a result of this step we will have modified the almost toric fibrations $\mu_i:X_i\to B_i$ and the map $\pi:\cX\to D_{\epsilon}$ so that 
    \begin{enumerate}
        \item The ATF's glue to a map $\mu:\cX_0\to \Delta_{\cX}$ which is well defined and continuous.  
        \item Denote by $P_{\gamma}$ the symplectic parallel transport from a regular fiber to the central fiber along some path. Then $\mu\circ P_{\gamma}$ extends smoothly to a proper Lagrangian submersion over the union of the interiors of the $2$-cells with small open neighborhoods of the $0$-cells. That is, over a set of the form $\Delta_{\cX}\setminus A$ where $A$ is a finite union of intervals contained in the interior of the edges of the 1-skeleton.
        \item The monodromy of the induced integral affine structure around loops in $\Delta_{\cX}\setminus A$ that surround no nodal point is trivial.
    \end{enumerate}
    In particular, because of the last item, there is a unique extension of the integral affine structure to $\Delta_{\cX}$.
    \item a \emph{hard step}: Each component $\delta$ of $ A$ is a closed interval in the interior of some $1$-cell $e_{ij}$. Fix a closed neighborhood $R\subset \Delta_{\cX}$ of $\delta$ so that 
    \begin{enumerate}
    \item $R$ is a rectangle in some affine coordinates
    \item  $\partial R\subset \Delta_{\cX}\setminus A$, and, 
    \item $R$ contains no nodal points.  
    \end{enumerate}
    Using classification results for fillings of contact tori, for  symplectic structures on $S^2 \times S^2$, and for symplectic log Calabi-Yau divisors in rational manifolds, we show there is a symplectomorphism $(P_{\gamma}\circ\mu)^{-1}(R)\simeq R\times \T^2$, and, moreover, this symplectomorphism intertwines the torus fibrations near $\partial R$. This allows to extend the torus fibration smoothly across $\delta$ after some modification near the pre-image of $\delta$. 
\end{itemize}

\subsection{Prospects for higher dimensions}
We expect the preparatory step to generalize after some weakening to toric degenerations in higher dimensions, dropping the requirement of a smooth total space. This time there will of course be monodromy around the singular locus.

The hard step relies heavily on hard theorems concerning symplectic topology in dimension four. One hopes that using Lefschetz fibration techniques one can reduce the dimension. Such ideas were employed by \cite{RuddatMak2021} in the construction of Lagrangians rational homology sphere on the mirror quintic threefold.

The work \cite{PelkaBobadilla} which appeared while this work was still in preparation suggests a complementary approach. Roughly speaking, \cite{PelkaBobadilla} construct a Lagrangian torus fibration of an open set in the A'Campo space over a neighborhood of some deformation of the 1-skeleton. This contrasts with our approach of starting with the union of the top stratum and neighborhoods of a finite number of points in the lower strata. 
\begin{rmk}
    One could try to prove Theorem \ref{Thm:LTFongoodKulikov} by interpolating the torus fibration of \cite{PelkaBobadilla} with the given ATF's. The details appear to us to be non-trivial.
\end{rmk}
\section*{Acknowledgements}
Y.G. and P.C. were supported by the ISF (grant no. 2445/20) and the BSF (grant no. 2020310). Y.G. is grateful to Umut Varolgunes for some helpful discussions.

PC is grateful to the Hebrew University of Jerusalem for a postdoctoral fellowship where a bulk of the work was undertaken. PC is thankful to Universit\'e Libre de Bruxelles where this project was completed. This work was supported by the FWO and the FNRS via EOS project 40007524. PC is grateful to Mohan Swaminathan and Martin Pinsonnault for helpful discussions and support during the completion of the project. PC is thankful to Jie Min for helpful discussions regarding constructing almost toric fibrations on the symplectic blow-up of rational manifolds. We are also very grateful to the anonymous referee whose report improved the document.  
\section{The preparatory step}

In this section, we construct the almost toric fibration over the complement of some open intervals in the intersection complex. We first define a piecewise nodal integral affine structure and construct a continuous fibration of the central fiber over it which is a nodal Lagrangian submersion in the top stratum. We then extend the piecewise nodal integral affine structure to a nodal integral affine structure. After a modification of the holomorphic map $\pi$ near each triple intersection point, the parallel transport defines a smooth map of the complement of some closed set in the general fiber to the intersection complex which is a proper nodal Lagrangian submersion over the complement of some intervals contained in the 1-skeleton of the intersection complex $\Delta_{\cX}$.  

Before proceeding, let us define the intersection complex $\Delta_{\cX}$ more carefully. The $1$-skeleton of $\Delta_{\cX}$ is a canonical construction: we have a $0$-cell (vertex) for each triple intersection point $X_{ijk}$ and a $1$-cell (edge) for each double intersection curve $X_{ij}$, connecting the two vertices corresponding to the triple intersection points it contains. However, when it comes to constructing the $2$-cells, we must make a choice. For each component $X_i$ of $\cX_0$ of a Kulikov model, the boundary divisor $\sum_j X_{ij}$ is a cycle of double intersection curves. We pick an arbitrary identification the corresponding cycle in the $1$-skeleton of $\Delta_{\cX}$ with the unit circle $S^1$. We use this identification to to attach a standard $2$-cell along this circle for each such cycle. The proof of Proposition~\ref{propPcwsNAffine} below suggests an alternative approach for constructing the $2$-cells using the almost toric fibrations $\mu_i$. That construction also involves a choice, namely the choice of an almost toric fibration on each component.

\subsection{Piecewise nodal integral affine structure on the intersection complex}

\begin{defn}
    A piecewise linear function $f:\R^m \rightarrow \R$ is a continuous function which is linear on some collection of polytopes $P_i$ such that $\bigcup_i P_i = \R^m$. A piecewise linear manifold $M$ is a topological space such that the transition functions are piecewise linear.
\end{defn}

\begin{defn}
Let $B$ be a (PL) smooth $2$-manifold. A \emph{(piecewise) nodal integral affine structure} on $B$ is a collection of isolated points $p_1,\dots p_N$ and a (piecewise) integral affine atlas on the complement $B_{reg}=B\setminus\{p_1,\dots, p_N\}$ which is integral affine on a punctured neighborhood of each $p_i$ and so that the monodromy around each nodal point $p_i$ is given in some integral basis by the shear matrix 
$$\begin{pmatrix}
1 & 1 \\
0 & 1
\end{pmatrix}$$
\end{defn}

\begin{exmp}
    \begin{itemize}
        \item The  $2$-torus, obtained by gluing opposite edges of a polygon, carries an integral affine structure.
        \item The interior of the base of an almost toric fibration carries a nodal integral affine structure.  
        \item The boundary of a polytope in $\R^3$ carries a piecewise integral affine structure.
     
    \end{itemize}
\end{exmp}

Fix almost toric fibrations $\mu_i:X_i\to B_i$ on the central fiber $\cX_0$.
\begin{prop}\label{propPcwsNAffine}
    The ATF's $\mu_i:X_i\to B_i$ glue together to a cell complex $B$ endowed with a piecewise nodal integral affine structure, and which is homeomorphic to the intersection complex $\Delta_{\cX}$.
\end{prop}
\begin{proof} 
 We construct a cell complex $B$ as follows. For each triple intersection point $X_{ijk}$ we introduce a $0$-cell $B_{ijk}$.  We fix an order on the $0$-cells, and for any double intersection component $X_{ij}$ we attach the interval $B_{ij}=[0,\omega(X_{ij})]$ to the two triple intersection points $X_{ijk},X_{ijk'}$ in an order preserving manner.  Finally, for the 2-cells we take the base diagrams $B_i$ and attach $\mu_i(X_{ij})$ to $B_{ij}$ by the unique affine map to the interval taking $\mu_i(X_{ijk})$ and $\mu_i(X_{ijk'})$ to $B_{ijk}$ and $B_{ijk'}$ respectively. It is clear that the construction is independent of the choice of order on the $0$-cells. Namely, if we change the order between a pair of successive vertices $X_{ijk}$ and $X_{ijk'}$ we get a map on the corresponding $1$-skeleta by mapping the interval $B_{ij}$ to itself via $t\mapsto \omega(X_{ij})-t$. This map extends to an isomorphism of the corresponding cell complexes. 

We obtain a homeomorphism between $B$ and the intersection complex $\Delta_{\cX}$ by mapping the 1-skeleton of $B$ to the $1$-skeleton of $\Delta_{\cX}$ in the obvious way. For each $i$ we then pick any homeomorphism from the $2$-cell associated with $X_i$ to $B_i$ in a way that commutes with the attaching maps. While this identification involves arbitrary choices, the piecewise nodal induced integral affine structures are all isomorphic. 
\end{proof}

We will consider the manifold $B$ constructed in the proof as a specific construction of the intersection complex and will not distinguish it from $\Delta_{\cX}$. On the other hand, we will not distinguish $B_i$ from its image in $B$. We refer to the union $\cup_i\partial B_i$ as the $1$-skeleton of $B$.

The  construction in Proposition \ref{propPcwsNAffine} induces a piecewise integral affine structure on any neighborhood $V$ of the $1$-skeleton that contains no nodes. That is, we have a strictly integral affine structure on each of the intersections $V\cap B_i$. We will later extend the  nodal affine structure from the complement of the 1-skeleton to the entire intersection complex. 

For now we introduce such an extension in a neighborhood of the $0$-skeleton. Consider a triple intersection point $p=X_{ijk}$. Let $v_1,v_2,v_3$ be primitive vectors tangent to the edges emanating from $p$ ordered as $ij<ik<jk$. This gives an identification $\eta_{ijk}$ of a small open neighborhood $V$  of $p$ in $B$ with an open neighborhood $V_0$ of the origin in $\partial (\R_{\geq 0})^3$. Namely $\eta_{ijk}$ is determined by the requirements 
\begin{itemize}
\item $\eta_{ijk}(v_l)=e_l$ for $l=1,2,3$,  and,
\item $\eta_{ijk}$ is linear on each component $V\cap B_l$.
\end{itemize}
We fix once and for all an integral affine structure on  $V_0$ by declaring the orthogonal projection $\pi_{V_0}$ whose kernel is spanned by the vector $e_1+e_2+e_3$ to be an integral affine chart. Clearly, this chart extends the integral affine structures on the components of $V_0$.

\begin{defn}\label{dfCanIntAff}
    Denote the image of $\pi_{V_0}$ by $\mathfrak{h}^*$ and let  $\phi_{ijk}:V\to \mathfrak{h}^*$ be the composition $\pi_{V_0}\circ \eta_{ijk}.$ We refer to the local integral integral affine structure generated by the chart $\phi_{ijk}$ as the \emph{canonical integral affine structure near the vertex $B_{ijk}$. }
\end{defn}
Note that this integral affine structure is independent of the order chosen on $i,j,k$.

\subsection{A fibration of the central fiber over the intersection complex}
Observe $\R^3$ is the dual Lie algebra of the standard 3-torus.  Considering the diagonal action of  $\T^3$ on $\C^3$, the subspace $\mathfrak{h}^*$ is the dual Lie algebra of the 2-torus which preserves the holomorphic volume form $\Omega_0:=dz_1\wedge dz_2\wedge dz_3.$ Denote by $\mu_0:\C^3\to \mathfrak{h}^*$ the moment map of this action. 

\begin{prop}\label{lmInterpolation1}
    After some modifications of each of the almost toric fibrations $\mu_i$  in a neighborhood of the $1$ dimensional stratum $\cup_{j} X_{ij}$, the $\mu_i$ glue together to a continuous map $\mu:\cX_0\to B$ which is an ATF over each $B_i\subset B$. Moreover, for each triple intersection point $X_{ijk}$, there is an open neighborhood $U$ of $X_{ijk}$ in $\cX$ and a symplectic embedding $\psi_{ijk}$ which, in the notation of Definition \ref{dfCanIntAff}, fits into a commutative diagram 
    \[
    \xymatrix{
        (U,\cup_{ij}X_{ij}\cap U) \ar[r]^{\psi_{ijk}} \ar[d]_{\mu} & (\C^3,\{z_1z_2z_3=0\})
        \ar[d]_{\mu_0} \\
        \mu(U\cap\cX_0)\ar[r]^{\phi_{ijk}} & \mathfrak{h}^*.
    }
\]
\end{prop}

Before proving this, we discuss interpolation of almost toric fibrations near a vertex.
\begin{defn}
    An admissible toric action on a neighborhood of the origin in $\C^2$ is a $\rho: \T^2 \rightarrow \Symp(U,\om_0)$ on an open neighborhood  $U$ of the origin in $(\C^2,\om_0)$ which is equivariantly symplectomorphic to the standard linear action $\rho_{st}$ with weights $(1,0)$ and $(0,1)$ at the origin by a symplectomorphism which preserves the coordinate stratification.  
\end{defn}
\begin{rmk}
    By the equivariant Darboux Theorem, this is equivalent to the assumption that $\rho$ fixes the origin, leaves the coordinate axes invariant, acts freely on the complement of the coordinate axes,  and the weights at the origin are $(1,0)$ and $(0,1)$.
\end{rmk}
\begin{lemma}\label{Lemma:InterpolatingtorusfibrationonC^2}
    Let $\rho_1: \T^2 \rightarrow \Symp(U,\om_0)$ denote an admissible toric action on an open neighborhood  $U$ of the origin in $(\C^2,\om_0)$. Then there exists a new admissible toric action $\rho:\T^2 \rightarrow \Symp(\C^2,\om_0)$ and open subsets $0 \in V_1\subset V_2 \subset \overline{V}_2 \subset U$  such that  $\rho|_{V_1} = \rho_1$ and $\rho|_{\C^2 \setminus V_2} = \rho_{st}$ where $\rho_{st}$ denotes the standard torus action. 
    
\end{lemma}
\begin{proof}
By assumption, there is an open neighborhood $V$ of the origin in $\C^2$ and a stratified equivariant symplectomorphism $\psi:V\to U$ where the action on $V$ is $\rho_{st}$. Without loss of generality, we may take $V$ to be the unit ball. We first observe that there exists a family $\psi_t:V=B_1\to \C^2$ of symplectic embeddings, such that $\psi_0$ is the inclusion, $\psi_1=\psi$, and such that $\psi_t$ preserves the stratification by the coordinate axes. Explicitly, we use Alexander's trick and consider the path $$
\psi_t(v):=\begin{cases}
        \frac1{t}\psi(tv),&\quad t\neq 0\\
        d\psi(0,0),&\quad t=0.
        \end{cases}
        $$
        The path can be shown to be smooth using Hadamard's Lemma. Note that as $\psi$ preserves the stratification, $d\psi(0,0)$ is a split symplectomorphism i.e a symplectomorphism for which the axes $\C\times \{0\}$ and $\{0\} \times \C$ are eigen directions for $d\psi(0,0)$. We may further connect the linear split symplectomorphism $d\psi(0,0)$ to the identity via a path of linear split symplectomorphisms. Let us still denote the concatenation by $\psi_t$.  We thus get a time dependent symplectic vector field $X_t$ with domain $\psi_s(V)$, by $p\mapsto \frac{d}{dt}\psi_{s+t}(p)|_{t=0}$. By simply connectedness we can find a time dependent Hamiltonian $H_t:\psi_t(V)\to \R$ so that $X_t$ is its Hamiltonian vector field. Let $\epsilon>0$ be such that $B_{\epsilon}\subset \psi_t(V)$ for all time $t\in[0,1]$. 

Let $f:\R^2\to[0,1]$ be a function which is identically $0$ for $x+y\geq\epsilon$, identically $1$ close enough to the origin, and such that $\frac {\partial f}{\partial x} $ is identically $0$ in a sufficiently small neighborhood of the $y$ axis, and $\frac {\partial f}{\partial y} $ is identically $0$ in a sufficiently small neighborhood of the $x$ axis. Let $g:\C^2\to\R$ be defined by $g(z_1,z_2):=f(|z_1|^2,|z_2|^2)$. Then the Hamiltonian flow of $gH_t$ is globally defined and is readily checked to preserve the axes. Let $\rho$ be the action defined by pulling back the standard action via the inverse of the time $1$ flow $\phi:\C^2\to\C^2$ of $gH_t$. Then $\rho$ satisfies all the requirements. Indeed, since $\psi$ fixes the origin, we verify that a small enough neighborhood of the origin is mapped for all $t\in[0,1]$ under  $\psi_t$ into the region where $f\equiv 1$. In that neighborhood we have $\phi=\psi.$ Admissibility of $\rho$ follows since $\phi$ preserves the stratification. 
\end{proof}

\begin{proof}[Proof of Proposition \ref{lmInterpolation1}]
For any collection of  ATF's $\mu_i:X_i\to B_i$ on the components of the central fiber $\cX_0$, the construction in Proposition \ref{propPcwsNAffine} gives a cell complex $B$ endowed with a piecewise nodal integral affine structure, and which is homeomorphic to the intersection complex $\Delta_{\cX}$. The maps $\mu_i$ do not a priori agree on the intersections $X_{ij},X_{jk},X_{ik}$. However,
we have for each triple intersections an equality of points $\mu_i(X_{ijk})=\mu_j(X_{ijk})=\mu_k(X_{ijk})$ and for each double intersection an equality of sets $\mu_i(X_{ij})=\mu_j(X_{ij})$. 
To obtain a continuous map on the central fiber we first modify the fibrations $\mu_i$ to make them agree as maps to the resulting $B$ on a neighborhood of each triple intersection $X_{ijk}$. To do so we first note that by symplectic orthogonality the configuration of divisors $X_{ij},X_{jk},X_{ik}$ is symplectomorphic near $X_{ijk}$ to the configuration of coordinate hyperplanes near the origin by a symplectomorphism $\psi_{ijk}$. 

We now apply Lemma \ref{Lemma:InterpolatingtorusfibrationonC^2} to modify $\mu_i$ on a neighborhood of $X_{ijk}$ in $X_i$ as follows. On the one hand, the almost toric fibration $\mu_i$ restricts to a toric action in an  open neighborhood $U$ of $X_{ijk}$ inside $X_i$. On the other hand, restricting to a smaller open neighborhood $V$ of $X_{ijk}$ inside $X_i$,  the symplectomorphism $\psi_{ijk}$ gives rise to an admissible toric action $\rho_1$ on $V$ by considering the moment map $\mu_0\circ\psi_{ijk}$. Applying Lemma \ref{Lemma:InterpolatingtorusfibrationonC^2} to this action $\rho_1$, we obtain a new admissible toric action $\rho$ on $U$ which agrees with $\rho_1$ on a smaller neighborhood $V_1$ of the origin and agrees with the original action outside a larger neighborhood $V_2\subset U$. After performing this construction for each of the three coordinate planes (corresponding to the three divisors meeting at $X_{ijk}$) and choosing appropriate affine coordinates on $B_i$, the modified $\mu_i$ satisfies $\mu_i=\mu_0\circ\psi_{ijk}$ on a neighborhood $V$ of $X_{ijk}$ in $X_i$. After this modification, the $\mu_i$ glue  to the continuous function $\mu_0\circ\psi_{ijk}$ on $V\cap\cX_0$. 

We now proceed to dealing with the edges. Note that $\mu_i|_{X_i \cap X_j}$ and $\mu_j|_{X_i \cap X_j}$ both induce circle actions on the sphere $S^2$ with the same two fixed points each of  weight $1$.  There is a  Hamiltonian $H_{ij}$ on $X_{ij}$ whose time 1-flow $\psi=\psi^{ij}$ satisfies $\mu_j\circ\psi|_{X_i \cap X_j} = \mu_i|_{X_i \cap X_j}$. This symplectomorphism can be constructed using the Arnold-Liouville normal form. Let $\Pi_i:N^{{\om_i}}_{X_i}(X_i\cap X_j) \rightarrow X_i\cap X_j$ denote the symplectic normal bundle of $X_i\cap X_j$ inside $X_i$. Consider $H_{ij}\circ\Pi$ on the normal bundle. By the symplectic neighbourhood theorem, this gives us an extension of $H_{ij}$ to a symplectic normal neighbourhood $V_i$ of $X_i \cap X_j$ in $(X_i,\om_i)$. Choose a bump function $\rho_i$ supported on $V_i$. Let $\psi^{ij}_i$ denote the time 1 flow of this new Hamiltonian $ \tilde{H}^{ij}_{i}:=\rho_iH_{ij}\circ\Pi$ on $X_i$. Then $({\psi^{ij}_i})^*\mu_i$ is a new almost toric fibration on $X_i$ which equals the original almost toric fibration $\mu$ outside of some open set. Furthermore, $({\psi^{ij}_i})^*\mu_i$ and $({\psi^{ij}_j})^*\mu_j$ both agree on $X_i \cap X_j$. Moreover, the new almost toric fibration agrees with the old one sufficiently close to the triple intersection points.  

\begin{rmk}
Note that although applying Lemma~\ref{Lemma:InterpolatingtorusfibrationonC^2} modifies the ATF $\mu_i$ and thus the integral affine structure on $B_i$, as the modification is via a symplectomorphism, the resulting IAS on $B_i$ is still integral affine isomorphic to the original one. 
\end{rmk}

\end{proof}

\subsection{The nodal integral affine structure on the intersection complex}
So far we have constructed a piecewise nodal integral affine structure on the intersection complex which is strict on the union of the interiors of the top strata with neighbourhoods of the vertices. We turn to show that this extends to a nodal integral affine structure everywhere. 
\tikzset{every picture/.style={line width=0.75pt}} 

\begin{figure}[H] 
    \centering
    \input{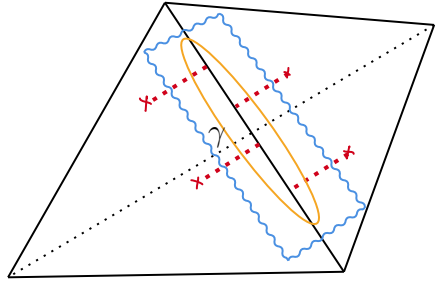}
    \caption{Monodromy around $\gamma$}\label{fig:NoMonodromy}
\end{figure}

\begin{thm}\label{tmNAffine}
    There is a unique extension of the piecewise nodal integral affine structure on the intersection complex introduced in Proposition \ref{propPcwsNAffine} to a nodal integral affine structure which agrees with the canonical structure of Definition \ref{dfCanIntAff} on neighborhoods of the vertices. 
\end{thm}
\begin{rmk}\label{rmkNoMonodromy}
      The crux of the statement amounts to triviality of the monodromy of the integral affine structure on the union of the complement of the 1-skeleton with neighborhoods of the vertices. See Figure~\ref{fig:NoMonodromy}. By a version of the Gauss-Bonnet Theorem in \cite[§6.5]{KontsevichSoibelman06} the statement would not be true for the case of the boundary of a Delzant polytope.   
\end{rmk}

\begin{proof}[Proof of Theorem \ref{tmNAffine}]
    We construct a chart in a neighborhood of each edge and show it agrees with the canonical structure near the vertices. Fix $e_{ij}:=\mu(X_{ij})$ an edge of the $1$ skeleton. Let $U\subset B_i\cup B_j\subset B$ be an open neighborhood of the interior of $e_{ij}$ containing no nodes. We consider continuous functions on $U$ whose restriction to each of $U\cap B_i$, $U\cap B_j$ is affine. To reduce the amount of freedom,  we fix a basepoint in the interior of $e_{ij}$ and we consider only the piecewise integral linear functions. Namely, piecewise integral affine functions which vanish at the basepoint and whose restriction to each side of the edge is affine.  The space $\mathcal{A}$ of all these is a free Abelian lattice of rank $3$. Indeed, a linear function is determined by two coefficients. $\mathcal{A}$ consists of pairs of linear functions that agree and the edge $e_{ij}$, and so there are $3$ degrees of freedom. Singling out an integral affine structure on $U$ which agrees with the piecewise nodal integral affine structure already given means declaring which of these piecewise linear functions are considered to be linear functions.
This amounts to specifying a rank 2 sublattice of $\mathcal{A}$ whose restriction to each side is the lattice of linear functions on that side near the edge.
We now define a homomorphism $wt:\mathcal{A}\to \Z$, the \emph{weight}, which will single out the desired sublattice. 

    To define the weight homomorphism, we observe that each element $f$ of $\mathcal{A}$ gives rise to a circle action $\rho_f$ on the restriction of the tangent bundle of the total space $T\cX|_{X_{ij}}$.  That is, an action on $X_{ij}$ which lifts to a linear action on the fibers. To see this, let $f_k=f_k|_{B_k}$ for $k=i,j$. Then $f_k\circ\mu$ is the Hamiltonian of a circle action on $X_k$ near the $1$-dimensional stratum. The linearization of this action gives circle actions on the vector bundles $TX_i|_{X_{ij}}$ and $TX_j|_{X_{ij}}$. The continuity equation $f_i|_{X_{ij}}=f_j|_{X_{ij}}$ implies that these two linearized actions agree on the common subbundle $TX_{ij}$, because on $TX_{ij}$ both are the differential of the same circle action on $X_{ij}$. Since $X_i$ and $X_j$ meet transversely along $X_{ij}$, we have a canonical identification
    \[
    T\cX|_{X_{ij}}\cong \left(TX_i|_{X_{ij}}\oplus TX_j|_{X_{ij}}\right)\Big/\{(v,-v):v\in TX_{ij}\}.
    \]
   The map $[u_i,u_j]\mapsto [\rho_i(\theta)\,u_i,\rho_j(\theta)\,u_j]$ for each $\theta\in S^1$ is well defined on the quotient and defines the required circle action $\rho_f$ on $T\cX|_{X_{ij}}$.

    Pick any point $p\in X_{ij}$ and any basis $v_1,v_2,v_3$ to $T_p\cX$ over the complex numbers. Let $\Omega$ be the Calabi-Yau form on $\cX$. Then we get map $\lambda_f:S^1\to S^1$ by $\theta\mapsto\arg\left(\Omega(\rho_f(\theta)\cdot (v_1\wedge v_2\wedge v_3))\right)$. Define $wt(f):= deg(\lambda_f)$. The following are easy to verify:
    \begin{itemize}
        \item The definition of $wt$ is independent of the choice of $p$ or the basis $\{v_1,v_2,v_3\}$ since any for any two choices the corresponding maps $\lambda_f$ are connected by a homotopy. It is also independent of the choice of $U$
        \item $wt:\mathcal{A}\to\Z$ is a homomorphism.
        \item $wt$ is surjective. 
    \end{itemize}

    We now show that $\ker(wt)$, when restricted to each side, gives exactly the lattice of linear functions on that side. Fix the side $U\cap B_i$ (the argument for $U\cap B_j$ is identical). Let $\mathcal A_i:=\{g\in\mathcal A: g|_{U\cap B_i}=0\}$; this is rank $1$. For nonzero $g\in\mathcal A_i$, the induced action is trivial on $TX_i|_{X_{ij}}$ but nontrivial on the normal line of $X_{ij}$ inside $X_j$, so $wt(g)\neq 0$ (for a primitive generator, $wt(g)=\pm 1$). Now let $\ell$ be any linear function on $U\cap B_i$, and choose $f\in\mathcal A$ with $f|_{U\cap B_i}=\ell$. Pick a primitive $g\in\mathcal A_i$. Then
    \[
    f':=f-wt(f)\,g
    \]
    still satisfies $f'|_{U\cap B_i}=\ell$, and $wt(f')=0$, so $f'\in\ker(wt)$. Hence every linear function on $U\cap B_i$ is the restriction of an element of $\ker(wt)$. Since $wt$ is surjective, $\ker(wt)$ is rank $2$. We therefore declare $\ker(wt)$ to be the lattice of linear functions, and this defines the integral affine structure on $U$.

    It remains to verify compatibility of this chart with the canonical integral affine structure near each vertex. This follows from the following observations:
    \begin{itemize}
        \item The integral affine functions for the canonical structure near a vertex can equivalently be characterized as being those whose induced action preserves the standard volume form on $\C^3$.
        \item For $p=X_{ijk}$, the volume form $\Omega_p$ equals the one induced from $\C^3$ by any local chart up to multiplication by a scalar, which does not affect the computation of the weight. 
    \end{itemize}
\end{proof}

\subsection{Continuity of parallel transport}\label{SubsecTransportContinuity}

Fix a point $x\in \bD$ and let $\gamma:[-1,0]\to \bD$ be a smooth path such that $\gamma(-1)=x$, $\gamma(0)=0\in\bD$ and such that $\gamma|_{[-1,0)}\subset\bD^*.$ By $\omega$-tameness, the map $\pi$ has symplectic fibers. Thus, for any $t\in[-1,0)$ we denote by $P_{\pi,\gamma,t}:\cX_x\to \cX_t$ the symplectic parallel transport. It is a symplectomorphism, and for homotopic paths with fixed endpoints, these symplectomorphisms are related by a Hamiltonian diffeomorphism. (See \cite{MS} pg.266 for more details on symplectic parallel transport).

We wish to analyze the behavior of parallel transport all the way to the singular fiber. It turns out, as we shall see, that for this to be well behaved at the lower dimensional strata of $\cX_0$ we need the symplectic form to be compatible there. Since $\omega$ is only assumed to satisfy this near the $0$-dimensional strata, we introduce the following notation. For each double curve $X_{ij}$, Let $C_{ij}\subset X_{ij}$ be a cylinder (homeomorphic to $S^1\times[0,1]$) with the property that 
\begin{itemize}
   \item $\omega$ is compatible with $J$ in a neighborhood of each point of $X_{ij}\setminus C_{ij}$. 
   \item The image $I_{ij}:=\mu_i(C_{ij})$ under the almost toric fibration $\mu_i$ is a closed sub-interval in the interior of the edge $e_{ij}\subset B$
\end{itemize}

We define the quotient spaces
\begin{align*}
\hat{\cX}_0 &:= \cX_0 \big/\!\!\sim, \quad\text{where } p\sim q \iff p=q \text{ or } \{p,q\}\subset C_{ij} \text{ for some } i,j,\\
\hat{B} &:= B \big/\!\!\sim, \quad\text{where } p\sim q \iff p=q \text{ or } \{p,q\}\subset I_{ij} \text{ for some } i,j.
\end{align*}
Let $\hat{p}_{ij}\in\hat{\cX}_0$ and $\hat{q}_{ij}\in\hat{B}$ denote the images of $C_{ij}$ and $I_{ij}$ under the respective quotient maps, and set $\hat{A}:=\{\hat{q}_{ij}\}$. The map $\mu_0:\cX_0\to B$ descends to a continuous map $\hat{\mu}_0:\hat{\cX}_0\to\hat{B}$. Moreover, $\hat{B}\setminus\hat{A}$ carries an integral affine structure and  $\hat{\mu}_0$ is an almost toric fibration over the complement $\hat{B}\setminus\hat{A}$.

\begin{lemma}\label{lmTrasprtContinuity1}
The map $P_{\pi,\gamma}:\cX_x\to\hat{\cX}_0$ defined by $z\mapsto[\lim_{t\to 0} P_{\pi,\gamma,t}(z)]$ is well defined and continuous. Moreover, 
on the preimage of any neighborhood of a nonsingular 
point of $\cX_0$ it is a 
symplectomorphism onto the image. Furthermore, the map $\Phi:\cX_{\gamma(-1)}\times[-1,0]\to\hat{\cX}_0$ defined by 
\[
\Phi(w,t):=\begin{cases}
P_{\pi,\gamma,t}(w) & \text{if }t\in[-1,0),\\
P_{\pi,\gamma}(w) & \text{if }t=0,
\end{cases}
\]
is continuous.
\end{lemma}
\begin{proof}
We establish well-definedness and continuity in the following three steps.

\medskip\noindent\textit{Step 1: Trapping neighborhoods.}
We show that for each $p\in\cX_0\setminus\bigcup_{i,j}C_{ij}$ and any open neighborhood $V\subset \cX$ of $p$ there exist an open neighborhood $U\ni p$ in $\cX$ with the following \emph{trapping property}: any parallel transport path $y(t)=P_{\pi,\gamma,t}(z)$ satisfying $y(t_0)\in U$ for some $t_0$ sufficiently close to $0$ remains in $V$ for all $t\in[t_0,0)$.

\emph{Case (a): $p$ is a smooth point of $\cX_0$.}
By shrinking $V$ we may assume that $\pi|_V$ is a submersion. The symplectic parallel transport along the reversed path $\bar\gamma(s)=\gamma(-s)$ is well defined on $\cX_0\cap V$. By smooth ODE theory (continuous dependence on initial conditions), paths starting in a sufficiently small neighborhood $W\subset \cX_0\cap V$ of $p$ remain in $V$ for sufficiently small time $t_0$. We may take $U$ to be the sweepout of $W$ by the flow parallel transport along $\bar\gamma|_{[t_0,0)}$.

\emph{Case (b): $p\in X_{ij}\setminus C_{ij}$ for some $i,j$.}
By definition of $C_{ij}$, there exists a neighborhood $V\ni p$ in $\cX$ on which $\omega$ is compatible with $J$.  Write $\pi=\pi_1+i\pi_2$ and let $\nabla\pi_1,\nabla\pi_2$ be the gradients with respect to the K\"ahler metric determined by $J$ and $\omega|_V$.  Since $\pi$ is holomorphic and $\omega|_V$ is compatible, these gradients span the symplectic complement of the fiber, giving the horizontal lift formula
\begin{equation}\label{eqHolHorLift}
y'(t)=\left(\gamma'_1(t)\nabla \pi_1(y(t))+\gamma'_2(t)\nabla \pi_2(y(t))\right)/|\nabla \pi|^2.
\end{equation}
By a direct calculation of the gradient of $z_1\dots z_n$, for a fibration $\pi$ with normal crossings singularities,
\begin{equation}\label{eqNormalCrossing}
|\nabla \pi|>C|\pi|^{1-1/n}
\end{equation}
for some constant $C=C(\pi)$ and $n$ the complex dimension, yielding
\begin{equation}\label{eqParrallelTransportBound}
    |y'|\leq|\nabla\pi||\gamma'(t)|/|\nabla\pi|^2<\frac1{C}|\gamma(t)|^{\left(\frac1{n}-1\right)}|\gamma'(t)|.
\end{equation}
This bound holds as long as $y(t)\in V$. Reparametrizing $\gamma$ as $t\mapsto t^ne^{i\theta(t)}$ makes the right-hand side bounded by a constant $C'$, so the total arc length of $y$ on $[t_0,0)$ is at most $C'|t_0|$ while the path stays in $V$. Take $V$ to contain the open ball of radius $2\epsilon$ around $p$ and $U$ the concentric ball of radius $\epsilon$. If $y(t_0)\in U$ and $t_0$ is close enough to $0$ that $C'|t_0|<\epsilon$, then $y$ cannot exit $V$ during $[t_0,0)$.

\medskip\noindent\textit{Step 2: Well-definedness.}
Fix $z\in\cX_x$ and let $y(t)=P_{\pi,\gamma,t}(z)$. Since $\pi^{-1}(\gamma)$ is compact, the $\omega$-limit set
$$\Lambda:=\bigcap_{s<0}\,\overline{\{y(t):t\in[s,0)\}}$$
is non-empty and connected. We claim $\Lambda$ maps under the quotient map to a single point of $\hat{\cX}_0$.

Suppose $p\in\Lambda\setminus\bigcup_{i,j}C_{ij}$, so $y(t_k)\to p$ for some $t_k\to 0^-$. By the trapping property, for $k$ large enough $y$ remains in $V$ for all $t\in[t_k,0)$ where $V$ is any open neighborhood of $p$. So, $\lim_{t\to 0} y(t)=p$.

If instead $\Lambda\cap(\cX_0\setminus\bigcup_{i,j}C_{ij})=\emptyset$, then $\Lambda\subset\bigcup_{i,j}C_{ij}$. Since the $C_{ij}$ are pairwise disjoint and $\Lambda$ is connected, $\Lambda$ lies in a single component $C_{ij}$. In either case, $\Lambda$ maps to a single point in $\hat{\cX}_0$. Hence $P_{\pi,\gamma}(z):=[\Lambda]\in\hat{\cX}_0$ is well-defined.

\medskip\noindent\textit{Step 3: Continuity.}
Fix $z\in\cX_x$ and let $p=P_{\pi,\gamma}(z)\in\hat{\cX}_0$. If $p$ lies in the top stratum of $\cX_0$, parallel transport is smooth near $p$ and continuity follows from standard ODE theory. 

For the lower strata, we distinguish two cases. If $p\in X_{ij}\setminus C_{ij}$ for some $i,j$, then the limit $\lim_{t\to 0}y(t)=p$ exists in $\cX_0$. Let $V$ be a neighborhood of $p$ in $\cX_0$. Let $V'$ be a neighborhood of $p$ in $\cX$ such that $ V=V'\cap \cX_0$. Let $U$ and $t_0$ be as in the trapping property for $V'$. By continuous dependence on initial conditions on $[x,t_0]$, for $z'$ sufficiently close to $z$ we have $P_{\pi,\gamma,t_0}(z')$ lands in $U$. Hence by the trapping property $\lim_{t\to 0} P_{\pi,\gamma,t}(z')\in V'\cap\cX_0=V$.

The remaining case is when $p=\hat{p}_{ij}$, meaning the $\omega$-limit set $\Lambda$ of $y(t)=P_{\pi,\gamma,t}(z)$ lies entirely in the compact cylinder $C_{ij}$. We argue by contradiction. Suppose $P_{\pi,\gamma}$ is not continuous at $z$. Then there exists an open neighborhood $\hat{A}$ of $\hat{p}_{ij}$ in $\hat{\cX}$ and a sequence $z_k \to z$ in $\cX_x$ such that for all $k$ we have $P_{\pi,\gamma}(z_k) \notin \hat{A}$. Let $A \subset \cX$ be the preimage of $\hat{A}$ under the quotient map. Since $C_{ij}$ is compact and $C_{ij} \subset A$, we can choose an open neighborhood $W$ of $C_{ij}$ such that $C_{ij} \subset W \subset \overline{W} \subset A$. Moreover, we can assume that $\partial W$ is disjoint from each $C_{kl}$.

Since the $\omega$-limit set $\Lambda$ of $y(t)$ is contained in $C_{ij}$, there exists a time $t_1 < 0$ such that $y(t) \in W$ for all $t \in [t_1, 0)$. Let $y_k(t) = P_{\pi,\gamma,t}(z_k)$. By continuous dependence on initial conditions of the parallel transport equation on the compact interval $[x, t_1]$, we have $y_k(t_1) \to y(t_1) \in W$. Thus, for all sufficiently large $k$, $y_k(t_1) \in W$.

However, since $P_{\pi,\gamma}(z_k) \notin \hat{A}$, the $\omega$-limit set of $y_k$ is not contained in $A$, and consequently not contained in $W$. Therefore, each path $y_k$ must eventually exit $W$. Let $s_k \in (t_1, 0)$ be the \emph{first} time $y_k(t)$ exits $W$, meaning $y_k(s_k) \in \partial W$. Because the sequence $s_k$ is bounded, we can pass to a subsequence such that $s_k \to s^* \in [t_1, 0]$. If $s^* < 0$, then $y_k(s_k) \to y(s^*)$ by continuous dependence on initial conditions for closed intervals in $[-1,0)$. Since $\partial W$ is closed, this would imply $y(s^*) \in \partial W$, which directly contradicts that $y(t) \in W$ for all $t \in [t_1, 0)$. Therefore, we must have $s^* = 0$.

Now consider the sequence of boundary points $y_k(s_k) \in \partial W$. Since $\pi(y_k(s_k)) = \gamma(s_k) \to \gamma(0) = 0$, properness of $\pi$ near $\cX_0$ implies that, passing to a further subsequence, $y_k(s_k)$ converges to some point $q \in \partial W \cap \cX_0$. Because $W$ is a neighborhood of $C_{ij}$, the point $q$ is disjoint from $C_{ij}$, so $q \in \cX_0 \setminus \bigcup_{k,l} C_{kl}$. Furthermore, since $\overline{W} \subset A$, we have $q \in A$.

Since $A$ is an open neighborhood of $q$, choose an open neighborhood  $V_q\subset A$ of $q$ satisfying the trapping property from Step~1 for some time $t_q < 0$. Since $y_k(s_k) \to q$ and $s_k \to 0$, for all sufficiently large $k$ we have $y_k(s_k) \in V_q$. The trapping property then guarantees that $y_k(t) \in V_q$ for all $t \in [s_k, 0)$, so the $\omega$-limit set of $y_k$ lies in $A$ and therefore $P_{\pi,\gamma}(z_k)\in\hat{A}$, contradicting our initial assumption. This completes the proof of continuity of $P_{\pi,\gamma}$.

\medskip\noindent\textit{Joint continuity.}
To prove continuity of $\Phi:\cX_{\gamma(-1)}\times[-1,0]\to\hat{\cX}_0$, it suffices to verify continuity at points of the form $(w,0)$, since $\Phi$ is continuous on $\cX_{\gamma(-1)}\times[-1,0)$ by smooth dependence of parallel transport on initial conditions. Fix $(w,0)$ and let $\hat{A}$ be an open neighborhood of $\Phi(w,0)=P_{\pi,\gamma}(w)$ in $\hat{\cX}_0$. The continuity of $P_{\pi,\gamma}$ established above provides an open neighborhood $U\subset\cX_{\gamma(-1)}$ of $w$ such that $P_{\pi,\gamma}(U)\subset\hat{A}$. By the same trapping arguments used in Step~3, there exists $\delta>0$ such that for all $w'\in U$ and $t\in[-\delta,0]$, we have $\Phi(w',t)\in\hat{A}$. This establishes continuity of $\Phi$ at $(w,0)$.

\end{proof}

For later use, we record the following description of neighborhoods of the points $\hat{p}_{ij}\in\hat{\cX}_0$ obtained by collapsing the cylinders $C_{ij}$.

\begin{lemma}\label{lmTubularHatPoint}
Fix a double curve $X_{ij}$ and the corresponding cylinder $C_{ij}\subset X_{ij}$, and let $\hat{p}_{ij}\in\hat{\cX}_0$ be its image under the quotient map $q:\cX_0\to\hat{\cX}_0$.  Let $N_{ij}\subset\cX$ be a tubular neighborhood of $C_{ij}$. For any open neighborhood $U\subset\cX_0$ of $C_{ij}$ satisfying $\overline{U}\subset N_{ij}$ there exists a constant $\delta>0$ such that for all
$t\in[-\delta,0)$ and all $z\in\cX_{\gamma(t)}$ satisfying
$$
\lim_{t\to 0} P_{\pi,\gamma|_{[t,0)},t}(z)\in\hat{U}=q(U)
$$
we have $z\in N_{ij}$.
\end{lemma}
\begin{proof}
Set $W:=P_{\pi,\gamma}^{-1}(\overline{\hat{U}})\subset\cX_{\gamma(-1)}$. By continuity of the map $\Phi:\cX_{\gamma(-1)}\times[-1,0]\to\cX$ and compactness of $W$, there is a $\delta>0$ such that $P_{\pi,\gamma,t}(w)\in N_{ij}$ for all $w\in W$ and $t\in[-\delta,0)$. The conclusion follows since $z\in\cX_{\gamma(t)}$ satisfies $\lim_{s\to 0}P_{\pi,\gamma|_{[s,0)},s}(z)\in\hat{U}$ if and only if $P_{\pi,\gamma,t}^{-1}(z)\in W$.
\end{proof}

\begin{lemma}\label{lmTrasprtContinuity2} 
    The statement of Lemma \ref{lmTrasprtContinuity1} remains true if we replace $\pi$ by another map $\tilde{\pi}:\cX\to\C$ which has symplectic fibers and satisfies the following conditions
    \begin{itemize}
        \item $$C|\pi|>|\tilde{\pi}|$$ for some $C>0$ and for all $x\in \cX,$
        \item 
        $$C|\nabla\tilde{\pi}|> |\nabla\pi|,$$
        \item  
         writing $\nabla\tilde{\pi}_1$, $\nabla\tilde{\pi}_2$ for the gradients of the real and imaginary part,
        $$\langle \nabla\tilde{\pi}_1, J\nabla\tilde{\pi}_2\rangle> \frac1{C}|\nabla\tilde{\pi}|^2. $$
    \end{itemize}
\end{lemma}
\begin{proof}
    We wish to obtain the estimate \eqref{eqParrallelTransportBound} with $y(t)$ the horizontal lift with respect to $\tilde{\pi}$. The first two estimates assumed in the statement of the present  lemma allow to deduce the estimate \eqref{eqNormalCrossing} for $\tilde{\pi}$ from the same estimate for $\pi$. 
    To obtain the middle estimate of \eqref{eqParrallelTransportBound}
    we first derive a formula for the horizontal lift for a non-holomorphic function. We claim that equation \eqref{eqHolHorLift} for the horizontal lift generalizes to
    \begin{equation}
        y'(t)=\frac{\gamma'_1(t)J\nabla \pi_2(y(t))+\gamma'_2(t)J\nabla \pi_1(y(t))}{\langle \nabla\tilde{\pi}_1, J\nabla\tilde{\pi}_2\rangle.}
\end{equation}
Indeed, one verifies this is symplectically orthogonal to $\ker d\tilde{\pi}$. Moreover, using $$
d\pi=\langle \nabla\tilde{\pi}_1, \cdot\rangle+i\langle \nabla\tilde{\pi}_2, \cdot\rangle
$$
one verifies it projects to $\gamma'(t)$ under $d\tilde{\pi}$.  We now readily deduce the estimate \eqref{eqParrallelTransportBound} for $\tilde{\pi}$. The rest of the proof is the same. 
    
\end{proof}

We can summarize the above results with the following corollary. 

\begin{cor}
   The map $\hat{\mu}_0\circ P_{\gamma}:\cX_{\gamma(-1)}\to\hat{B}$ is continuous for any $\tilde{\pi}$ satisfying the estimates of Lemma \ref{lmTrasprtContinuity2}. Moreover, the restriction to the pre-image of the interior of any face $B_i$ is an almost toric fibration.
\end{cor}

\subsection{Smoothness of parallel transport}
Since the central fibre has normal crossing singularities the parallel transport map is generally not a smooth map for points that flow to lower dimensional strata of $\cX_0$, even where it's continuous. Our goal in this section is to modify $\pi$ to some $\tilde{\pi}$  so that $\hat{\mu}_1=\hat{\mu}_0\circ P_{\gamma}$ is smooth also near the nodal points. For this we use the following idea that has been attributed to Ruan (\cite{Ruan:2001}).

\begin{lemma}\label{Lemma:Ruansidea}
Let $(M,\om)$ be a symplectic manifold. Let $\pi:M \rightarrow \C$ be a projection with symplectic fibres. Let $H$ be a Hamiltonian function on $(M, \om)$ and suppose the  Hamiltonian flow of $H$ preserves the fibers of $\pi$. Then $H$ is preserved by symplectic parallel transport. That is
$$
H\left(P_t(x)\right)=H(x)
$$
for all $x \in M, t \in \R$, where $P_t$ is the time-t parallel transport symplectomorphism along any curve $\gamma$ in $\C$.
\end{lemma}

Motivated by this, we proceed as follows. First, we discuss a \emph{local model for a vertex}. Consider the function $\pi_0:\C^3\to \C$ given by $\pi_0(z_1,z_2,z_3)=z_1z_2z_3$.  The function $\pi_0$ is invariant under the holomorphic volume-preserving torus action whose moment map is $\mu_0$.

Call a local symplectic chart near a triple intersection point $X_{ijk}$ admissible if it fits into the diagram of Proposition \ref{lmInterpolation1}. 
 
\begin{lemma}\label{Lemma:Interpolationoffibration}
    There exists a modification $\tilde{\pi}$ of $\pi$ such that 
   \begin{itemize}
        \item $\tilde{\pi}$ has symplectic fibres,
        \item for each triple intersection point  $X_{ijk}$ there is a local  admissible Darboux chart $\psi: U\subset \C^3\to V\subset \cX $ which satisfies $\pi_0 =\tilde{\pi}\circ \psi,$
    \item
    $\tilde{\pi}$ agrees with $\pi$ outside some symplectic balls,
    \item $\tilde{\pi}^{-1}(0)=\pi^{-1}(0)$,
    \item $\tilde{\pi}$ satisfies the estimates of Lemma
    \ref{lmTrasprtContinuity2}.
    \end{itemize}
\end{lemma}
\begin{proof}

    For any triple intersection point we fix an admissible local symplectic chart $\psi: U\subset\C^3\to \cX$. Consider the function $g=\pi\circ \psi-\pi_0.$ For each $I\subset \{1,2,3\}$ introduce the notation  $\pi_I=\Pi_{i\in I}z_i\Pi_{i\not\in I}\overline{z}_i$. 
    We claim $g$ can be written WLOG in the form
    $$
    g=\sum_{I\subset\{1,2,3\}} \pi_I h_I
    $$
    where $h_I:U\to\C$ are smooth function which vanish at the origin at least to order $1$. To see this let $\psi_{hol}:U\to \cX$ be holomorphic coordinates pulling back $\pi$ to the function $\pi_0$. These exist by the normal crossing assumption. WLOG $\pi$ is normalized so $d\psi_{hol}=d\psi$. Let $f=\pi\circ\psi$ then
    $f=\pi_0\circ (\psi_{hol}^{-1}\circ \psi).$ Let $\phi_i$ be the $i$th component of  $\psi_{hol}^{-1}\circ \psi$. Since $\psi_{hol}^{-1}\circ \psi=Id+O(|z|^2)$, we have $\phi_i(z_1,z_2,z_3)=z_i+O(|z|^2)$.  Moreover, since $\phi_i(z_1,z_2,z_3)-z_i$ vanishes identically on the plane $z_i=0$ we actually have $\phi_i=z_i+z_ia_i+\overline{z}_ib_i$ for $a_i,b_i:U\to\C$ are smooth functions which are $O(|z|)$. Rearranging $f=\phi_1\phi_2\phi_3$, we get the desired form.

    Let $\rho:\R_+\to[0,1]$ be a bump function identically $0$ near the origin and identically $1$ near $1$. For $r\in(0,1)$ let $\rho_r(z_1,z_2,z_3)=\rho(|z|/r).$  We claim that for $r$ small enough, the function $\pi_r=\pi_0+\rho_rg$ has symplectic fibres. For this recall that according to \cite{Donaldson}, the criterion for having symplectic fibres is $|\overline{\partial}\pi_r|<|\partial \pi_r|$.

    Let $C$ be an estimate for $\max_{I\subset \{1,2,3\}}\{\partial h_I,\overline{\partial}h_I\}$ on some precompact neighborhood of the origin in $U$ and for $\rho'$ on the unit interval. Then for sufficiently small $r$ we get 
    $$|\partial\pi_r|\geq |\partial \pi_0|-\sum_I\left(|\rho h_I\partial \pi_I|+C|\rho\pi_I|\right)- \left|\sum h_I\pi_I\right|C/r$$ and, $$|\overline{\partial}\pi_r|\leq \sum_I\left(|\rho h_I\overline{\partial} \pi_I|+C|\rho\pi_I|\right)+\left|\sum h_I\pi_I\right|C/r
    .$$
    All the terms can be made arbitrarily small relative to  $|\partial \pi_0|$ by considering small enough neighborhood of the origin. Indeed, we have $\frac{|\partial\pi_I|}{|\partial \pi_0|}$ $\frac{|\overline{\partial}\pi_I|}{|\partial \pi_0|}$ are $O(1)$,  $h_I$ and $\frac{|\pi_I|}{|\partial \pi_0|}$  are $O(|z|)$ and $\frac{|h_I\pi_I|}{|\partial \pi_0|}$ are $O(|z|^2)$. The estimates of Lemma    \ref{lmTrasprtContinuity2} are also immediate. 
\end{proof}

\subsection{Conclusion of the preparatory step}
We summarize the construction so far in the following lemma.
\begin{prop}\label{propSummaryOfPrepStep}
    There are modifications of the almost toric fibrations $\mu_i$ on the central fiber and of the function $\pi$ so that 
    \begin{enumerate}
      
    \item the $\mu_i$ glue to a continuous function $\mu_0:\cX_0\to \Delta_{\cX}=B$,
    \item Writing $\hat{\cX}_0$ and $\hat{B}$ for the quotient spaces constructed in Lemmas~\ref{lmTrasprtContinuity1} and~\ref{lmTrasprtContinuity2}, and $\hat{\mu}_0:\hat{\cX}_0\to\hat{B}$ and $P_{\pi,\gamma}:\cX_{\gamma(-1)}\to\hat{\cX}_0$ for the induced almost toric fibration and limiting parallel transport map, the composite
    \[
        \hat{\mu}_1:=\hat{\mu}_0\circ P_{\pi,\gamma}:\cX_{\gamma(-1)}\to \hat{B}
    \]
    is continuous,
    \item $\hat{\mu}_1$ is an almost toric fibration when restricted to $\hat{\mu}_1^{-1}(\hat{B}\setminus\hat{A})$, where $\hat{A}=\{\hat{q}_{ij}\}$ is the finite set of collapsed points in $\hat{B}$ corresponding to the cylinders $\{C_{ij}\}$,
    \item the Arnold-Liouville integral affine structure on $B\setminus A=\hat{B}\setminus\hat{A}$ coincides with the restriction of the one from Theorem \ref{tmNAffine}. 
    \end{enumerate}
\end{prop}

\begin{proof}
    (1) is Proposition~\ref{lmInterpolation1}. We then modify $\pi$ in accordance with Lemma~\ref{Lemma:Interpolationoffibration}. The modified map satisfies all the properties of Lemma \ref{lmTrasprtContinuity2} by construction; by the corollary to Lemma~\ref{lmTrasprtContinuity2}, $\hat{\mu}_1 = \hat{\mu}_0\circ P_\gamma$ is therefore continuous, proving (2). 
    
    For (3), that $\hat{\mu}_1$ is an almost toric fibration over the interior of the top stratum follows from Proposition~\ref{lmInterpolation1}. Near each vertex $B_{ijk}$, Lemma~\ref{Lemma:Interpolationoffibration} provides an equivariant symplectomorphism $\psi_{ijk}$ of a neighborhood of $X_{ijk}$ with a neighborhood of the origin in $\C^3$ intertwining $\tilde{\pi}$ with $\pi_0(z_1,z_2,z_3)=z_1z_2z_3$. The moment map $\mu_0$ of the standard torus action on $\C^3$ preserving the holomorphic volume form is invariant under the flow of $\pi_0$, hence by Lemma~\ref{Lemma:Ruansidea} is preserved by parallel transport. Thus $\hat{\mu}_1=\mu_0\circ\psi_{ijk}^{-1}$ near the preimage of $B_{ijk}$, giving an almost toric fibration over a neighborhood of $B_{ijk}$ whose Arnold-Liouville coordinates coincide with the canonical affine structure $\phi_{ijk}$ from Definition~\ref{dfCanIntAff}. 
    
    As the modifications in Proposition~\ref{lmInterpolation1} preserve the integral affine structure near the vertices, the Arnold-Liouville structure on $\Delta_{\cX}\setminus A$ coincides with the restriction of the one from Theorem \ref{tmNAffine}, proving (4).
\end{proof}
The following Corollary is useful in Section~\ref{Sec:HardStep}.

\begin{cor}\label{Cor:NoMonodromy}
 Pick a curve $\gamma$ in $B\setminus A$  so that the region bounded by $\gamma$ in $B$ contains no nodal points. Then there is a neighbourhood $U$ of $\gamma$ such that $\mu^{-1}_1(U)$ is symplectomorphic to $U \times T^2$.
\end{cor}
\begin{proof}
     By Remark~\ref{rmkNoMonodromy}, the monodromy of the integral affine structure around the loop $\gamma$ is the identity.  
\end{proof}

\section{The hard step}\label{Sec:HardStep}

In this section, we complete the proof of Theorem~\ref{Thm:LTFongoodKulikov} by extending the almost toric fibration over the remaining parts of the base. The strategy consists of three main steps:

\begin{enumerate}
    \item First, we show that near each point of $\hat{A}$, we can find a neighborhood $\hat{R}\subset\hat{B}$ whose preimage $R\subset B$ under the quotient map is a rectangular neighborhood (in the sense of the integral affine structure on $B$) and such that the preimage $\hat{\mu}_1^{-1}(\hat{R})$ is diffeomorphic to $R \times \T^2$ (Lemma~\ref{lmDiffeo}). This relies on Wendl's classification of minimal strong fillings of the standard contact structure on $\T^3$.
    
    \item Second, we show the pre-image $\hat{\mu}_1^{-1}(\hat{R})$ is symplectomorphic to $R \times \T^2$. This is done by compactifying to $S^2\times S^2$ and using classification results for symplectic structures and log Calabi-Yau divisors. 
    
    \item Finally, we need to interpolate the torus fibration $\hat{\mu}_1$ defined over a neighborhood of $\partial \hat{R}=\partial R$ with the one defined over $R$ coming from a symplectomorphism $\hat{\mu}^{-1}_1(\hat{R})\simeq R\times \T^2$.
\end{enumerate}

\subsection{Proof of Theorem~\ref{Thm:LTFongoodKulikov}}

Let $q_B:B\to\hat{B}$ denote the quotient collapsing each interval $I_{ij}$ to the point $\hat{q}_{ij}$. Note that the pre-image under $q_B$ of each point $\hat{c}$ of $\hat{A}$ lies in the interior of an edge of the $1$-skeleton of $B$. Given a neighborhood $\hat{R}$ of $\hat{c}$ in $\hat{B}$, we denote by
$$
R:=q_B^{-1}(\hat{R})\subset B.
$$ We call $R$ a \emph{rectangular neighborhood} of $q_B^{-1}(\hat{c})$ in $B$ if $R$ is integral affine isomorphic to a product of intervals in the plane. In this case, we will also say that $\hat{R}$ is a neighborhood of $\hat{c}$ with rectangular preimage $R$.

\begin{lemma}\label{lmDiffeo}
    Let $\hat{R}\subset\hat{B}$ be a closed neighborhood of a point $\hat{c}\in\hat{A}$ 
    with rectangular pre-image $R\subset B$. If $\hat{R}$ is contained in a sufficiently small neighborhood of $\hat{c}$ then the preimage of $\hat{\mu}_1^{-1}(\hat{R})$  is diffeomorphic to $R\times \T^2$. Moreover, the diffeomorphism can be taken to intertwine $\hat{\mu}_1$ with $q_B\circ \mathrm{pr}_R$ close enough to the boundary, where $\mathrm{pr}_R$ is the projection of $R \times \T^2$ onto the first factor.
\end{lemma}
The proof of Lemma \ref{lmDiffeo} relies on \cite[Corollary 4]{Wendl} which we recall after introducing the necessary terminology. A symplectic 4-manifold $(W,\om)$ is a \emph{strong filling} of a contact 3-manifold $(M,\xi)$ if $\partial W = M$ and there exists a vector field $Y$ defined near $\partial W$ which is transverse to the boundary and satisfies $L_Y \om = \om$ and $\xi = \ker \iota_Y \om$. We further say that $(W,\om)$ is a \emph{minimal strong filling} of $(M,\xi)$ if $W$ contains no symplectic 2-spheres with self-intersection -1. 

\begin{thm}(Corollary 4 in \cite{Wendl})\label{Thm:Wendl}
Every minimal strong filling of $(\T^3,\xi_{\std})$ is diffeomorphic to $\T^2 \times D$ where $D$ is the 2 dimensional disk and where $\xi_{\std}$  is the standard contact structure on $\T^3$ induced by the form $\alpha:= \cos({2\pi z})dx + \sin({2 \pi z}) dy $.
\end{thm}

\begin{lemma}\label{lmRMinimalFilling} 
    $\hat{R}$ as in Lemma \ref{lmDiffeo} can be thickened to a neighborhood $\hat{R}'$ so that 
    $\hat{\mu}_1^{-1}(\hat{R}')$ is a minimal strong filling of  $(\T^3,\xi_{\std})$. 
\end{lemma}
\begin{proof}
    Fix a tubular neighborhood $N$ of the cylinder $C=\hat{\mu}_0^{-1}(\hat{c})$. Let $U\subset \overline{U}\subset N\cap\cX_0$ and assume $R$ is such that $\pi^{-1}(R)\subset U$. Thicken to $\hat{R}'$ so that $R':=q^{-1}_B(\hat{R}')$ is convex with smooth boundary and so that $\hat{\mu}_0^{-1}(\hat{R}')\subset U$. Then  Corollary \ref{Cor:NoMonodromy} immediately implies that $\hat{\mu}_1^{-1}(\hat{R}')$ is strong symplectic filling of $(\T^3,\xi_{\std})$. For minimality it suffices to prove that $\hat{\mu}_1^{-1}(\hat{R}')$ is topologically aspherical. For this note that by Lemma \ref{lmTubularHatPoint} the image of $\hat{\mu}_1^{-1}(\hat{R})$ under parallel transport to $\cX_{\gamma(t)}$ for $t$ close enough to $0$ is contained in the normal neighborhood $N$ of $C$ which a cylinder (the sphere with two little discs removed) in the total space of the fibration. But this parallel transport is a diffeomorphism as long $t$ does not reach $0$. 
\end{proof}

For our application, we need a slight strengthening of Theorem \ref{Thm:Wendl} which is also contained in \cite{Wendl}. For a parameter $\sigma$ \cite[\S5.1]{Wendl} introduces a model filling $W_{\sigma}$ of $T^*\T^2\setminus [-1,1]^2\times \T^2$. In particular, for the standard contact form on the unit cotangent circle bundle $W_{\sigma}$ contains the end  $([2,\infty)\times T^3, de^\rho\alpha)$ for $\alpha$ the standard form and $\rho$ the radial coordinate. On the other hand, let $W$ be a minimal strong filling of $(\T^3,\lambda)$ for $\lambda=\iota_Y\omega$ inducing the standard contact structure. Let $W_C$ be the result of attaching to $W$ a trivial symplectic cobordism with lower boundary $(\T^3,\lambda)$ and upper boundary $e^C\alpha$. Let $W_{\infty}$ be obtained from  $W_C$ by attaching the symplectization end $([C,\infty)\times T^3, de^a\alpha)$. Note the model filling $W_{\sigma}$ naturally contains $([C,\infty)\times T^3, de^a\alpha)$ for $R$ large enough. For appropriate choice of $\sigma$ \cite[Proposition 5.6]{Wendl} declares that for $C'>C$ large enough we have a diffeomorphism $W_{\sigma}\to W_{\infty}$ which is identity on 
$[R_1,\infty)\times T^3$. On the other hand, by the construction of $W_{\sigma}$ in \cite[\S 5.1]{Wendl} there is a diffeomorphism $T^*\T^2\to W_{\sigma}$ which is the identity on $[C,\infty)\times T^3$. We summarize this discussion with the following corollary.
\begin{cor}\label{cor:WendlEnd}
For a strong minimial filling $W$ of $(\T^3,\lambda)$ and $W_{\infty}$ as in the previous paragraph, there is a diffeomorphism $T^*\T^2\to W_{\infty}$ which is the identity outside a compact set.
\end{cor}

\begin{proof}[Proof of Lemma \ref{lmDiffeo}]
      Letting $W=\hat{\mu}_1^{-1}(\hat{R}')$ as in Lemma \ref{lmRMinimalFilling} and attaching a cone to obtain $W_{\infty}$, Corollary \ref{cor:WendlEnd} gives a diffeomorphism $T^*\T^2\to W_{\infty}$ which intertwines $\hat{\mu}_1$ with the projection to the cotangent fiber outside of a compact set. The claim about $\hat{\mu}_1^{-1}(\hat{R})$  follows by inflating a neighborhood of $\hat{\mu}_1^{-1}(\partial \hat{R})=\mu_1^{-1}(\partial R)$ in a fiber preserving manner.
\end{proof}

Theorem~\ref{Thm:LTFongoodKulikov} will now be an immediate consequence of the following proposition.
\begin{prop}\label{prpHard}
    Let $\om_0$ denote the standard symplectic form on $T^*\T^2=\R^2\times \T^2$ and let  $\mu_{\std}:T^*\T^2\to\R^2$ denote the projection to the first factor. Let $R\subset\R^2$ be a rectangle. Let $\om$ be a symplectic form on  $R\times \T^2 \subset T^*\T^2$ such that on an open neighborhood $U$ of the boundary $\partial \mu_{\std}^{-1}(R)$ we have $\om|_U=\om_0|_U$. Then there exists a diffeomorphism $\Psi: R \times \T^2 \rightarrow R \times \T^2$ and an open neighborhood $U'\subset U$ of $\partial \mu_{\std}^{-1}(R)$ such that $\Psi^*\om_0 =\om$ and $\mu_{\std}\circ \Psi=\mu_{\std}$  on $U'$.
\end{prop}
\begin{proof}[Proof of Theorem~\ref{Thm:LTFongoodKulikov} assuming Proposition \ref{prpHard}]
By the last item of Proposition \ref{propSummaryOfPrepStep} there is no monodromy around the components of $A$. See Remark \ref{rmkNoMonodromy}. Thus each point $\hat{c}$ of $\hat{A}$ admits a neighborhood $\hat{R}_{\hat{c}}$ in $\hat{B}$ whose preimage $R_{\hat{c}}\subset B$ is rectangular in the above sense. By Lemma~\ref{lmDiffeo} there is a diffeomorphism $\phi:\hat{\mu}_1^{-1}(\hat{R}_{\hat{c}})\to R_{\hat{c}}\times \mathbb{T}^2$. We can pick $\phi$  which agrees near $\partial \hat{\mu}_1^{-1}(\hat{R}_{\hat{c}})$ with Arnold-Liouville coordinates induced by $\hat{\mu}_1$. In particular, $\phi$ is a symplectomorphism near the boundary. Let $\Psi:R_{\hat{c}}\times \mathbb{T}^2\to R_{\hat{c}}\times \mathbb{T}^2$ be the result of applying Proposition \ref{prpHard} to the symplectic form $(\phi^{-1})^*\omega$ (viewing $R_{\hat{c}}$ as a rectangle in $\R^2$ via an integral affine chart). Then $\Psi\circ\phi:\hat{\mu}_1^{-1}(\hat{R}_{\hat{c}})\to R_{\hat{c}}\times\mathbb{T}^2$ is a symplectomorphism and the composition $\mu':=\mu_{\std}\circ\Psi\circ\phi$ agrees with $\hat{\mu}_1$ on a neighborhood of $\partial \hat{\mu}_1^{-1}(\hat{R}_{\hat{c}})$ (where $\mu_{\std}$ is defined to be the standard projection map onto $R_{\hat{c}}$ in $R_{\hat{c}} \times \T^2$). We thus define 
$$
\mu(b)=\begin{cases}
    \hat{\mu}_1(b),&\quad b\in \hat{B}\setminus\cup_{\hat{c}}\hat{R}_{\hat{c}}=B\setminus \cup_{\hat{c}}R_{\hat{c}}\\
    \mu'(b),&\quad b\in \cup_{\hat{c}}R_{\hat{c}}.
\end{cases}
$$
By construction in Proposition \ref{propSummaryOfPrepStep}, the IAS induced by $\hat{\mu}_1$ agrees with the one from Theorem \ref{tmNAffine}.  Similarly, by construction, the IAS induced by $\mu'$ is the standard one on $R_{\hat{c}}\times\T^2$ with $R_{\hat{c}}$ an affine rectangle in $\Delta_{\cX}$ with the same IAS. Thus $\mu$ is as required.

\end{proof}

\subsection{Proof of Proposition \ref{prpHard}}
For the proof of Proposition \ref{prpHard}, we recall the following theorems.


\begin{thm}(Symplectic Torelli Theorem for rational surfaces, {Theorem 13 in \cite{Li-Min-Ning}})\label{Theorem:Isotopyofdivisors}
 Let $(M_1,\om_1)$ and $(M_2,\om_2)$ be a pair of rational symplectic 4-manifolds. Let $D_1,D_2$ be a pair of symplectic log Calabi-Yau divisors in $M_1$, $M_2$ with components $D_{1,i},D_{2,i}$ respectively. Suppose the components of $D_1$ and $D_2$ are $\omega_1$-orthogonal and $\om_2$-orthogonal respectively. If there is an integral isometry $$
\gamma: H^2(M_1 ; \mathbb{Z}) \rightarrow H^2(M_2 ; \mathbb{Z})
$$ which maps $PD([D_{1,i}])$ to $PD([D_{2,i}])$, and whose real extension $\gamma_R : H^2(M_2;\R)\rightarrow H^2(M_1;\R)$ maps $[\om_1]$ to $[\om_2]$, then $(M_2,\om_2,D_2)$ is isomorphic to $(M_1,\om_1,D_1)$. 
    
\end{thm}

We also need the following proposition whose proof is given in Section \ref{seInterpolation}.

\begin{prop}\label{prpATFinterpolation}
    Let $\mu:X\to B$ be an ATF with boundary cycle $D=\cup D_i$. Let $U\subset X$ be an open neighborhood of $D$, let $V\subset B$ be a neighborhood of $\partial B$, and let $\mu': U\to V$ be an ATF with the same boundary cycle $D$ and inducing the same integral affine structure. Then there exists a symplectomorphism $\Phi:X\to X$ so that $\mu\circ \Phi$ agrees with $\mu'$ on a smaller neighborhood of $D$.
\end{prop}

\begin{proof}[Proof of Proposition \ref{prpHard} assuming Proposition \ref{prpATFinterpolation}]
    Compactify $R \times \T^2$ to $S^2 \times S^2$ by symplectic reduction on the boundary rectangle. Let $D:= B_1 \cup F_1 \cup B_2 \cup F_2$ denote the configuration of boundary divisors. Let $\om,\omega_0$ still denote the induced symplectic forms on the compactification. Here $B_1$ and $B_2$ are in the class of the base $B= [S^2 \times \{*\}]$ and $F_1$ and $F_2$ are in the class of the fibre $F:= [\{*\} \times S^2]$.

    Since $\omega$ coincides with $\omega_0$ on $U$ we have that the induced form on $S^2\times S^2$ still coincides with a standard product form in a neighborhood of the boundary divisors. In particular, these are all symplectically orthogonal. Furthermore, as the Chern class of the tangent bundle is determined by its restriction to a neighbourhood of the boundary divisors of  $S^2 \times S^2$, we also have that the first Chern classes are equal for both $\omega$ and $\omega_0$, and hence the boundary divisor $D$ is log Calabi-Yau for both the forms. By the same reasoning, we have $[\omega]=[\omega_0]$. By Theorem~\ref{Theorem:Isotopyofdivisors}, there exists a diffeomorphism  $\psi:S^2\times S^2\to S^2\times S^2$ and so that $\psi^*\omega_0=\omega$ and  $\psi(D)=D$. Restricting the diffeomorphism to the complement of $D$ gives us a diffeomorphism, still denoted by $\psi$ and mapping $R \times \T^2\to R\times \T^2,$ such that $\Phi^*\om_0=\omega$. 
   
   It remains to modify $\Phi$ further so it preserves $\mu_{\std}$ near the boundary. This is achieved by postcomposing $\Phi$ with an appropriate $\Psi$ as in Proposition \ref{prpATFinterpolation}. 

\end{proof}

\subsection{Interpolation of ATF's}\label{seInterpolation}

We turn to prove Proposition \ref{prpATFinterpolation}. The proof will require some observations from symplectic topology.

\begin{prop}\label{Lemma:Isotopytoidentityonaneighbourhood}
Let $(M,\om)$ be a symplectic manifold and let $S\subset M$ be a symplectic submanifold. Let $\Phi$ be a symplectomorphism of an open neighbourhood $U$ of $S$ onto an open neighbourhood of $S$. Suppose $\Phi|_S = \id$ and $d\Phi$ restricted to $TM|_S$ is the identity as well. 
Then $\Phi$ is Hamiltonian isotopic to the identity on a possibly smaller neighbourhood $U'$ of $S$. 
Furthermore, if $\Phi$ is the identity on an open neighbourhood $V$ in $M$ of a closed subset $B\subset S$ 
then the Hamiltonian can be taken to be $0$ on an open neighbourhood of $B$. 
\end{prop}
\begin{proof}
  Consider the graph $\Gamma(\Phi)$ of $\Phi$ in $(U \times M, \om \oplus (-\om))$. $\Gamma(\Phi)$ is a properly embedded Lagrangian submanifold of $U\times M$. By the assumption regarding $\Phi|_S$ and $d\Phi$ restricted to $TM|_S$ is the identity, by taking arbitrarily small neighborhoods $U'$ of $S$, the intersection of $\Gamma(\Phi)$ with neighborhoods $U'\times M$ of $S\times S$ is arbitrarily $C^1$ close  to the diagonal Lagrangian $\Delta\simeq U'$. Identifying a Weinstein neighborhood of $\Delta$ with an open neighborhood of the $0$ section in $T^*U'$, the Lagrangian $\Gamma(\Phi)$ becomes the graph of a closed $1$-form $\alpha$ over $U'$. The latter vanishes on $S$ by assumption. We may assume $U'$ deformation retracts to $S$. Therefore, the class of $\alpha$ vanishes, and in particular, $\Phi$ is Hamiltonian isotopic to the identity. Moreover, the primitive $f$ is constant along $S$ and may be taken to be $0$ there. Suppose $\Phi$ is identity along an open neighborhood of $A$. We may pass to a smaller neighborhood which deformation retracts to a subset of $S$. Since the Hamiltonian vanishes on $S$, it vanishes on that neighborhood as well.

\end{proof}

\begin{lemma}\label{Lemma:Makingthenormalidentity}
Let $(M,\om)$ be a symplectic $4$-manifold and $S\subset M$ a  symplectic sphere. Let $U$ be an open neighbourhood of $S$ and $\phi:U\rightarrow M$ a symplectic embedding whose restriction $\phi|_S$ is the identity. Then there exists a Hamiltonian isotopy $\beta_t:M \rightarrow M$ such that $\beta_t|_S = id$ and on  $d\beta_1 = d\phi$ on $S$. Moreover, if there is a closed subset $A\subset S$ for which $\phi$ is the identity on an open neighborhood of $A$ in $M$, then $\beta$ can be taken to be the identity on some open neighborhood of $A.$
\end{lemma}

The proof of Lemma \ref{Lemma:Makingthenormalidentity} relies on the following Lemma from \cite{Michor}.

\begin{lemma}
    
[Relative Poincaré Lemma] \label{RelativePoincareLemma}
Let $M$ be a smooth finite dimensional manifold, let $N \subset M$ be a closed submanifold, and let $k \geq 0$. Let $\omega_t$ $0\leq t\leq1$ be a 1-parameter family of closed $(k+1)$ forms on $M$ which vanishes when pulled back to $N$. Then there exists a 1-parameter family of $k$-form $\varphi_t$ on an open neighborhood $U$ of $N$ in $M$ such that $d \varphi_t=\omega_t|_ U$ and $\varphi=0$ along $N$. If moreover $\omega_t=0$ for all $t$ on $TM|_N$, then we may choose $\varphi_t$ such that the first space derivatives of $\varphi_t$ vanish on $N$.
    
\end{lemma}
\begin{proof}
The proof is exactly as in Lemma 31.16 in \cite{Michor} rephrased for 1-parameter families. 
\end{proof}
\begin{proof}[Proof of Lemma \ref{Lemma:Makingthenormalidentity}]
Since $d\phi$ preserves the tangent space to the symplectic submanifold $S$, it also preserves the normal space, and so defines a symplectic automorphism of $N(S)$. The structure group $Sp(2, \mathbb{R})$ of $N(S)$ is   homotopy equivalent to $U(1)$. Since $U(1)$ is Abelian, one deduces from this that the space of symplectic automorphisms of $N(S)$ is homotopy equivalent to the space of maps $S^2\to S^1$ which connected, since $\pi_2(S^1) = 0$. Therefore, we can isotope the automorphism $A:= d\phi|_{N(S)}$ to the identity via a path $A_t$ of fiberwise symplectic automorphisms. Consider the 1-parameter family of diffeomorphisms $\phi_t:= \exp \circ A_t \circ \exp^{-1}$ where $\exp$ is the exponential map from the normal bundle to an open neighbourhood of $S$. Note that $d\phi_t|_{N(S)} = A_t$. Consider the 1-parameter family of symplectic forms $\om_t:= \phi_t^*\om$. As $A_t$ are symplectic automorphisms, we see that $\om_t|_{T_S M} = \om$ for all $t$ (where $T_S M$ denotes the restriction of $TM$ to $S$).  By Lemma~\ref{RelativePoincareLemma}, we can find a 1-parameter family of primitives $\sigma_t$ such that $d\sigma_t = \frac{d} {dt}\om_t$ and $\sigma_t$ vanishes along $S$. Mimicking the proof of Moser's theorem we then obtain a time dependent vector fields $X_t$ whose flow gives us a path of diffeomorphisms $\psi_t$ such that $\psi_t^*\om_t =\om$ and moreover $d\psi_t|_{T_S M} = \id$. Then $d(\phi_t \circ \psi_t)|_{T_S M} = A_t$ and $\phi_t\circ\psi_t$ is a path of symplectomorphisms connected to the identity. 
Let $Y_t$ be the symplectic vector field which is the time derivative of  $\phi_t\circ\psi_t$. 
Then in a small simply connected neighbourhood $U$ of the sphere $S$, $Y_t$ is Hamiltonian and is generated by a function $H_t$. Considering the time 1-flow of the bumped-off Hamiltonian $\rho H_t$ gives us the required global symplectomorphism $\beta: M \rightarrow M$ whose derivative by construction on the normal of $S$ is $A_1 = d\phi|_{N(S)}$. For the last sentence in the statement, note that the path $\phi_t$ is identity on a neighborhood of $B$ by construction. Therefore, $\sigma_t$ can be chosen to vanish on a neighborhood of $B$.  
\end{proof}

\begin{prop}\label{Prop:Dehntwist}
   Let $S \subset M$ be a symplectically embedded sphere in a symplectic 4-manifold $M$. Suppose $M$ is equipped with a Hamiltonian torus action for which $S$ is invariant. Let $U_x$ and $U_y$ denote small neighbourhoods in $M$ of the toric fixed points $x,y$ respectively. Given any symplectomorphism $\phi:M \rightarrow M$ such that $\phi(S)=S$ and $\phi|_{U_x} = \phi|_{U_y} = \id$ there exists an open neighborhood $U$ of $S$ and local symplectomorphism $\psi: U \rightarrow U$ such that 
    \begin{enumerate}
        \item $\psi$ preserves the moment map $\mu$
        \item $\psi$ is identity on $U_x\cup U_y$
        \item $\phi \circ \psi$ restricted to $S$ is isotopic to the identity via symplectomorphisms of $S$ that are identity on $U_x\cap S$ and $U_y\cap S$.
    \end{enumerate}
\end{prop}
\begin{proof}
 Let $\om_0$ denote the standard symplectic form on the sphere $S^2$ and let $\Omega_{x,y}$ denote the space of symplectic forms on $S^2$ that are in the same cohomology class as $\om_0$ and are equal to the standard symplectic form in some fixed small neighbourhood $V_x$ and $V_y$ of the points $x,y \in S^2$ respectively. Note that all elements of $\Omega_{x,y}$ tame the standard complex structure. Thus $\Omega_{x,y}$ is convex. By Moser's theorem we have the following fibration 
 \[
\Symp_{\id,x,y}(S^2) \rightarrow \Diff_{\id,x,y}(S^2) \twoheadrightarrow \Omega_{x,y}(S^2)
 \]
where $\Diff_{\id,x,y}$ denotes the space of diffeomorphisms that are the identity on $V_x$ and $V_y$ and $\Symp_{\id,x,y}$ are the symplectomorphisms that are the identity on the same open neighbourhoods. As $\Omega_{x,y}$ is  contractible, we get a homotopy equivalence $\Symp_{\id,x,y} \simeq \Diff_{\id,x,y}$. 
By Proposition 2.4 in  \cite{FarbMargalit}, we know that $\pi_0(\Diff_{\id,x,y}) \simeq \Z$ and furthermore each component contains a representative of a Dehn twist. As the Dehn twists are in fact symplectomorphisms identity on $V_x$ and $V_y$, the above discussion implies that each isotopy class of $\Symp_{\id,x,y}$ is represented by a Dehn twist. 

By assumption $\phi|_S \in \Symp_{\id,x,y}$. Let $\Lambda_n$ denote the Dehn twist isotopic to $\phi|_S$. Since $\Lambda_n$ is identity on $U_x\cap S$ and $U_y\cap S$ we can pick a trivialization of the normal bundle over the complement of $U_x,U_y$, assuming WLOG they are invariant, and extend $\Lambda_n$ to an equivariant symplectomorphism of  the normal bundle as $\tilde{\Lambda}_n (x,v):= (\Lambda(x),v)$ for all $(x,v) \in (S\setminus (U_x\cup U_y)) \times \R^2$. We take $\psi=\Tilde{\Lambda}^{-1}$. Clearly, $\psi$ maps a neighborhood $U$  to itself and preserves $\mu$.  Moreover, $\phi \circ \psi$ is isotopic to the identity via symplectomorphism that are identity on $U_x$ and $U_y$. 
 
\end{proof}

\begin{proof}[Proof of Proposition \ref{prpATFinterpolation}]
By picking Lagrangian sections $\sigma,\sigma'$ of $\mu,\mu'$ respectively over $V$,  Arnold-Liouville provides us with a symplectic map $\Phi_0:U\to X$ such $\mu'=\mu\circ\Phi_0$ over $V$. The claim will follow if we show that the map $\Phi_0$ is Hamiltonian, at least after composing with a fiber preserving local symplectomorphism. Indeed, bumping off the Hamiltonian away from the boundary divisors, we obtain a map $\Phi:X\to X$ which agrees with $\Phi_0$ near the boundary divisors. This $\Phi$ satisfies the requirements of the Proposition. 

To proceed, let $D_i$ denote the irreducible components of the boundary divisor $D$. Let $D_{ij}:= D_i \cap D_j$ ($i \neq j$). It is automatic that $\Phi_0(D_{ij})= D_{ij}$. Using the argument in the proof of \ref{Lemma:InterpolatingtorusfibrationonC^2} we obtain a Hamiltonian diffeomorphism $\Psi_{vert}:X\to X$ so that $\Theta_{vert}:= \Psi_{vert}\circ\Phi_0$ is the identity on a neighbourhood of the points $D_{ij}$ and still preserves $D$. Namely, we symplectically identify a small ball $B_{ij}$ centered at  $D_{ij}$  with the standard ball in $\C^2$ so that $D_i,D_j$ map to the coordinate axes. We use the Alexander trick to connect $\Phi^{-1}|_{B_{ij}}$ to the identity via symplectomorphisms which preserve the coordinate axes, and then modify to get a global symplectomorphism which is supported in $\cup_{i,j}B_{ij}$, preserves the divisor $D$ and agrees with $\Phi_0^{-1}$ on $B_{ij}$ (after slightly shrinking). The details are exactly the same as in the proof of Lemma \ref{Lemma:InterpolatingtorusfibrationonC^2}. 

For each $i$, let $\psi_i:U\to U$ be a map satisfying the properties guaranteed by Proposition \ref{Prop:Dehntwist} with respect to $\Theta_{vert}\circ\psi_i$. We may extend $\psi_i$ as identity in a neighborhood of the divisors $D_j$ for $j\neq i$. We can take this neighborhood to still be $U$. Let $\psi$ be the composition of the $\psi_i$ taken in any order. Then   $\Theta_{vert}\circ\psi|_{D_i}$ is isotopic to the identity via symplectomorphisms that are identity near a fixed neighbourhood of $D_{ij}$. Since $\psi$ preserves $\mu'$, we adjust our initial choice of map $\Phi_0$ to $\Phi_0\circ\psi$. We thus assume WLOG that $\Theta_{vert}|_{D_i}$ is isotopic to the identity via symplectomorphisms that are identity near a fixed neighbourhood of $D_{ij}$.

We extend this Hamiltonian isotopy of $\Theta_{vert}|_{D_i}$to a neighborhood of $D$  using the normal projection. Call the time $1$ map of this isotopy $\alpha$. Then $\Theta_{nbhd}:=\alpha^{-1}\circ\Theta_{vert} =\alpha^{-1}\circ \Psi_{vert}\circ\Phi_0$ is a local symplectomorphism which is identity on $D$ and on a neighborhood of the $D_{ij}$ and is Hamiltonian isotopic to $\Phi_0$.

Lemma \ref{Lemma:Makingthenormalidentity} gives us an additional Hamiltonian symplectomorphism $\beta:X\to X$ such that $d\beta|_{D_i} = d\Theta_{nbhd}|_{D_i}$. Let $\Theta_{der}:=\beta^{-1}\circ \Theta_{nbhd}$. The local symplectomorphism $ \Theta_{der}$ is Hamiltonian isotopic to $\Phi_0$ and satisfies $\Theta_{der}|_{D_i} = \id,$ $d\Theta_{der}|_{D_i} = \id,$ and $\Theta_{der}$ is the identity near $D_{ij}.$

Thus the restriction of $\Theta_{der}$ to a neighborhood of any of the components $D_i$ satisfies the conditions of Proposition \ref{Lemma:Isotopytoidentityonaneighbourhood} with $A$ a neighbourhood of the vertices $D_{ij}$. Therefore, $\Theta_{der}$ is Hamiltonian in that neighborhood and the Hamiltonian generating it can be taken to be $0$ in a neighbourhood of $\bigcup_{i,j}D_{ij}$. Thus $\Theta_{der}$ is Hamiltonian isotopic to the identity on a neighborhood of $D$. Since $\Phi_0$ is Hamiltonian isotopic to $\Theta_{der}$, the claim follows.


\end{proof}

\section{Symplectic Kulikov models from toric degenerations}\label{Sec:KulikovfromFano} 

In this section we prove Theorem \ref{thmToricKulikov}. The proof proceeds by first constructing a suitable resolution of the singular locus of the toric degeneration, and then showing that this resolution admits a symplectic form with the required properties. Finally, we compute the integral affine structure. 

Let  $X$ be a smooth complex 3-dimensional toric Fano variety endowed with a toric symplectic form $\omega$. Let $P\subset \mathfrak{t}^*$ be the moment polytope, and let $\mu_X: X \rightarrow P$ denote the moment map. We denote by $D_i$ the toric boundary divisors and by $F_i$ the corresponding faces of $P$.  The toric boundary $\sum D_i$ of $X$ is the zero locus of a section $s_0$ of the anti-canonical bundle $-K_X$. We fix a smooth hypersurface $H$ in $X$ cut out transversally as the zero locus of another section $s_1$ of $-K_X.$ This will be our generic fiber. 

Consider the  family $\pi: E \rightarrow \CP^1$ where $E$ is defined as  
\begin{equation}\label{Equation:Pencilofsections}
  {E}:= \left\{\left(x,[\lambda_0:\lambda_1]\right) \subset X \times \CP^1~|~ \lambda_0s_0(x) + \lambda_1 s_1(x)=0\right\}, 
\end{equation}
and $\pi$ is the restriction of the projection from $X \times \CP^1 \rightarrow \CP^1$ to $E$. Our aim is to modify this family into a symplectic Kulikov model. The first problem we are faced with is that the total space $E$ has singularities. 
\subsection{Singularities of the total space and their resolution}
\begin{defn}{(Base locus)}
The base locus of the pencil defined by $s_0,s_1$ is the set of points on $x \in X$ such that $s_0(x)=s_1(x)=0$. 
\end{defn}

\begin{lemma}\label{Lemma:Baselocusintersectsallstrata}
For any choice of $s_1,$ the base locus of the pencil defined by $s_0,s_1$ intersects each connected component of the codimension 2 toric stratum. 
\end{lemma}
\begin{proof}
    Let $C_{ij}$ be the curve  $D_i \cap D_j$. 
    The anti-canonical bundle is ample. Therefore, so is its pullback to $C_{ij}$. So $-K_X|_{C_{ij}}$ has strictly positive degree. It follows that the restriction of $s_1|_{C_{ij}}$ has a zero. 
 
\end{proof}
\begin{lemma}\label{Lemma:Singularpointsnearcentralfibre}

    Assume $H:= s_1^{-1}(0)$ intersects each positive dimension toric stratum transversely\footnote{According to \cite[\S6.6]{Mikhalkin2004} this holds if $H$ is the closure of a very affine hypersurface with Newton polytope $P$ close enough to the tropical limit.
    }. 
    Then the singular points of the total space $E$ 
    are the points in the fiber over $\lambda_1=0$ which are the intersection of the base locus with the codimension 2 toric strata.
\end{lemma}

\begin{proof}
In a neighborhood of the point at infinity  $\lambda_0=0$, the defining  equation can be written as $\lambda s_0+s_1=0$ for $\lambda$ the coordinate $\lambda_0/\lambda_1$. In particular on the fiber over infinity, $s_1=0$, and the differential coincides with $ds_1+s_0d\lambda$ which has full rank by smoothness of $H$ and independence of $d\lambda$ from $ds_1$. So, we focus on the complement of the fiber over infinity. Set $\lambda_0=1$ and write now $\lambda=\lambda_1$. The differential is now $ds_0+\lambda ds_1+s_1d\lambda$. Since $d\lambda$ is independent of $ds_0,ds_1$, singular points can only occur when $s_1=0$ which implies also $s_0=0$. Since $H$ meets each component transversely, the relation $ds_0+\lambda ds_1=0$ can only hold if $\lambda=0$ and therefore $ds_0=0$. So, singular points occur only at intersections of $H$ with higher codimension strata of the toric boundary. By Lemma \ref{Lemma:Baselocusintersectsallstrata} these are precisely the intersections with the $1$-dimensional strata.

\end{proof}
 
From now on we assume $H$ satisfies the hypothesis of Lemma \ref{Lemma:Singularpointsnearcentralfibre}. Let $p$ be a singular point of the total space $E$ occuring as the intersection of $H=s_1^{-1}(0)$ with a codimension $2$ stratum $S$. The stratum $S$ is the is the intersection of a pair of divisors $D_1,D_2$ which are locally given by certain regular functions $z_1=0$ and $z_2=0$. In a neighborhood of $p$ in the ambient space $X\times\CP^1$, these functions together with $z_3=\lambda$ and $z_4=s_1$ form a local coordinate system in which $p$ is given by the equation $z_1z_2=z_3z_4.$ Such a singularity is called an ordinary double point singularity. It admits a resolution by \emph{small blow up} wherein $p$ is replaced by a copy of $\CP^1$, rather than by a codimension $1$ variety. Moreover, there are two distinct such resolutions. To describe this write $\mathcal{Z}:= \left\{z_1z_4 - z_2z_3 = 0\right\}\subset\C^4$. The two resolutions are as follows. 

\[
\mathcal{Z}^1_{sb}:= \left\{\left( (z_1,\cdots, z_4),[u;v]\right) \in \C^4 \times \CP^1 ~ | ~ z_1v= z_3u~~~ z_2v=z_4u \right\} 
\]
\[
\mathcal{Z}^2_{sb}:= \left\{\left( (z_1,\cdots, z_4),[u;v]\right) \in \C^4 \times \CP^1 ~ | ~ z_1v= z_2u~~~ z_3v=z_4u \right\} 
\]

A more geometric description is as follows. Consider the proper transform $\tilde{\mathcal{Z}}$ of $\mathcal{Z}$ in the blow up of $\C^4$ at the origin.  The exceptional locus $\tilde{\mathcal{Z}}$ is a copy of $\CP^1 \times \CP^1$.  Contracting along one of the rulings yields a construction of the small blow-up described above. The choice along which $\CP^1$ to contract gives the two small blow-ups. 

There is a natural projection map 
\begin{align*}
    \pi_{sb}: \mathcal{Z}_{sb} &\longrightarrow \mathcal{Z} \\
    ((z_1,\cdots, z_4),[u;v]) &\mapsto (z_1,\cdots, z_4)
\end{align*}
which is a biholomorphism away from the exceptional locus $((0,\cdots, 0),[u;v])$.

\subsection{Symplectic form on the small blow up}
The local model just described can be implanted to resolve any ordinary double point singularity in a variety $E$. The small blow up is naturally endowed with an integrable complex structure. However, even if the original manifold is K\"ahler away from the double points, the small blow up does not give rise to a K\"ahler structure. We would like to endow the small blow up with a symplectic form taming the complex structure. This is only possible under some topological conditions formulated in \cite{Conifold}. We recall the details. 

We restrict attention to the following setup, Let $(M,J)$ be a projective complex 3-dimensional variety with ordinary double point singularities at finitely many points $p_1 \cdots, p_k$ and a holomorphic embedding into a smooth projective toric variety $i: (M,J) \rightarrow Y$. We endow $M$ with a symplectic form $\om := i^*\om_{Y}$ which is the pull-back of a toric K\"ahler form on $Y$. Such a form is K\"ahler away from the nodes on $M$ and we refer to this as a \emph{singular projective K\"ahler variety} \footnote{\cite{Conifold} introduce a notion of nodal singularity of a symplectic manifold which allows to work without the restriction to projective varieties. We do not need this generality here however.} Let $(M_{sb},J_{sb})$ denote a choice small blow-up of $(M,\om,J)$ at the singular points and let $\pi_{sb}: M_{sb} \rightarrow M$ denote the standard projection.  

\begin{lemma}\label{Lemma:Family_of_Form_small_blowup}
 Assume there exists a $4$-cycle $D\in C_*(M_{sb};\R)$ which is a real linear combination of smooth cycles and such that the homological intersection satisfies $[D] \cdot [C] > 0$ for any exceptional sphere $C$.  Then there exists a representative $\alpha_{D}$ of the Poincar\'e dual to $[D]$ and an $\epsilon>0$ such that for any $t\in(0,\epsilon)$ the form $\omega_{tD}:= \pi_{sb}^* \om +t\alpha_{D}$ is  symplectic and  tames the complex structure $J_{sb}$.  
\end{lemma}
\begin{rmk}

    Note that the symplectic area of an exceptional sphere $C$ with respect to $\omega_{tD}$ is given by the homological intersection number $[tD] \cdot [C] $.
\end{rmk}
\begin{proof}
     The proof is given as part of the proof of Theorem~2.9 in \cite{Conifold}. We recall the details for self containment. Let $\tilde{\alpha}_D$ denote a 2-form in the  class  Poincar\'e dual to $D$.  For each exceptional curve $C$ fix a holomorphic isomorphism of the normal bundle to $C$ with the holomorphic vector bundle to $\mathcal{O}(-1) \oplus \mathcal{O}(-1)$. In a tubular neighborhood $U_C$ we can pull back the standard K\"ahler form  on the total space of the bundle $\mathcal{O}(-1) \oplus \mathcal{O}(-1)$ for which the area of $\CP^1$ is $1$. Shrinking the neighborhood $U_C$ if necessary, the pulled back form $\omega_C$ tames $J_{sb}$ on $U_C$. Let $\lambda_C:=[D]\cdot[C]>0$. Then the form $\lambda_C\omega_C$ still tames $J_{sb}$ on $U_C$. Moreover, since the homology of $U_C$ is generated by $[C]$, the closed form $\tilde{\alpha}_D-\lambda_C\omega_C$ is exact. We can thus interpolate $\tilde{\alpha}_D$ with $\lambda_C\omega_C$ by cutting off a primitive. We obtain a closed two form $\alpha_D$ which agrees with $\lambda_C\omega_C$ in a sufficiently small neighborhood of each exceptional curve $C$. More importantly, it tames $J_{sb}$ in an open neighborhood $V$ of the exceptional locus. By compactness of $M_{sb}\setminus V$ the closed two-form $\pi_{sb}^*\om+t\alpha_D$ is non-degenerate for $t$ small enough. On the other hand, since both forms tame $J_{sb}$ on $V$, we have that $\pi_{sb}^*\om+t\alpha_D$ is non-degenerate on $V$ for any $t$.

\end{proof}
Our aim to apply this to the singular projective K\"ahler variety $E$ embedded into $Y=X\times \CP^1$ and an appropriate choice of small blow up $E_{sb}$. For such a choice let $\Tilde{\pi}: E_{sb} \rightarrow \C$ be the composition $\pi \circ \pi_{sb} $. We start with the symplectic form $i^*\omega_Y$ induced from the inclusion of $E$ into $X\times\CP^1$.

It should be noted that unlike the usual blowup which can be realized by a surgery, and thus doesn't modify the symplectic form away from the blow-up locus, the small blow up involves a global modification of the symplectic form. In particular, the symplectic form on the regular fibers will also be affected. Our aim is to construct a cycle $D$ satisfying the condition of  Lemma \ref{Lemma:Family_of_Form_small_blowup} such that the symplectic form induced on regular fibers is closely related to the original one.

We first need to specify how to carry out the blow up. There are $N=24$ double points, and for each of them we have a binary choice. We will consider a scheme for making these choices we call \emph{greedy}. Pick any order on the facets $F_1,\dots F_m$ of the moment polytope $P$.  Proceed inductively as follows. For $F_1$ we take all double points occurring on the boundary of $F_1$ and carry out our blow up so the divisor $D_1$ is the one that gets blown up for each of them. Then at the $i$th stage we consider only the edges that have not yet been blown up, and, if there are any points not yet resolved, we blow them up inside $D_i$. From now in we refer by $E_{sb}$ to the complex manifold obtained by this procedure. 

 For each face $F_i$ let $B_i\subset s_0^{-1}(0)$ be the component of the base locus meeting $D_i,$ let $$\Sigma_i=\left\{\left(x,[\lambda_0:\lambda_1]\right) \in E~|~ x\in B_i\right\},$$
and let $\tilde{\Sigma}_i\subset E_{sb}$ be  the proper transform of $\Sigma_i.$ 
\begin{lemma}\label{lmIntersection} Pick a strictly decreasing sequence of positive real $t_1> t_2\dots> t_m>0$. The cycle $\Sigma_{t_1,\dots, t_m}:=\sum_{i=1}^mt_i\tilde{\Sigma}_i$ has positive intersection with each exceptional curve $C$. More precisely, if the underlying double point of $C$ occurs in an intersection $D_{ij}=D_i\cap D_j$ for $m\geq j>i$, the intersection number is $[\Sigma_{t_1,\dots, t_m}]\cdot [C] =t_i-t_j$. 
\end{lemma}

\begin{proof}
    The first claim follows from the second. To prove the latter let us focus on the local model near one singular point $p$. In a neighborhood of the double point  $p$ the nodal projective variety $E$ is described up to units by the equation $z_iz_j=\lambda s_1$ where $z_i,z_j$ define the divisors $D_i,D_j$, and $\lambda,s_1$ are as before.  Then $\tilde{\Sigma}_i$ is the proper transform of $\{z_i=s_1=0\}$ which intersects $C$ transversally and thus intersects $C$ with intersection number $+1$ by holomorphicity. Similarly, $\tilde{\Sigma}_j$ is the proper transform of $\{z_j=s_1=0\}$ which contains $C$ and thus has intersection number $-1$ with $C$. See Figure \ref{fig:resolution_intersection}.

\end{proof}

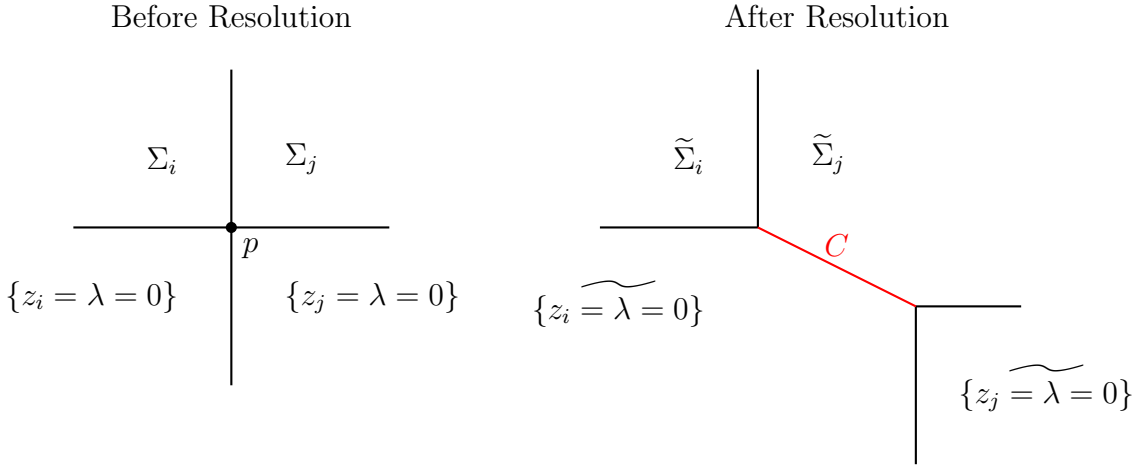
\begin{figure}[h]
    \centering
    \begin{tikzpicture}[scale=1.4]
        \begin{scope}[xshift=-2.5cm]
            \draw[thick] (0,0) -- (1.5,0);
            \draw[thick] (0,0) -- (0,1.5);
            \draw[thick] (0,0) -- (-1.5,0);
            \draw[thick] (0,0) -- (0,-1.5);
            
            \node[above right] at (0.4,0.4) {$\Sigma_j$};
            \node[above left] at (-0.4,0.4) {$\Sigma_i$};
            \node[below left] at (-0.4,-0.4) {$\{z_i=\lambda=0\}$};
            \node[below right] at (0.4,-0.4) {$\{z_j=\lambda=0\}$};
            
            \fill (0,0) circle (1.5pt);
            \node[below right] at (0,0) {$p$};
            
            \node[above] at (0,1.8) {Before Resolution};
        \end{scope}
        
        \begin{scope}[xshift=2.5cm]
            \draw[thick] (-1.5,0) -- (0,0);
            \draw[thick] (0,1.5) -- (0,0);
            
            \draw[thick] (1.5,-0.75) -- (2.5,-0.75);
            \draw[thick] (1.5,-2.25) -- (1.5,-0.75);
            
            \draw[thick, red] (0,0) -- (1.5,-0.75) node[midway, above] {$C$};
            
            \node[above right] at (0.4,0.4) {$\widetilde{\Sigma}_j$};
            \node[above left] at (-0.4,0.4) {$\widetilde{\Sigma}_i$};
            \node[below left] at (-0.4,-0.4) {$\widetilde{\{z_i=\lambda=0\}}$};
            \node[below right] at (1.8,-1.2) {$\widetilde{\{z_j=\lambda=0\}}$};
            
            \node[above] at (0.75,1.8) {After Resolution};
        \end{scope}
    \end{tikzpicture}
    \caption{Local model of the resolution at a double point $p$ defined by $z_iz_j=\lambda s_1$. 
    }
    \label{fig:resolution_intersection}
    \end{figure}

\begin{rmk}\label{Remark:Modifiedform_SymplecticOrthogonal}
    Note that the modification to the symplectic form is done by adding the Poincar\'e dual of the base locus cycle which doesn't meet the codimension 2 strata of the singular fibre. Therefore it can be taken to be supported away from the codimension 2 strata in the blow up. As a result we have that the proper transforms of the codimension 1 strata $D_{ij}$ and $D_{ik}$ and $D_{jk}$ still intersect symplectically orthogonally with respect to the modified symplectic form as well. Furthermore, the symplectic form is in fact compatible with the $J_{sb}$, in a small neighbourhood of the codimension 2 strata thus satisfying the requirements of Defintion~\ref{Defn:AdmissibleKulikovModel}.
\end{rmk}
In the following lemma we identify the a  fiber of $\tilde{\pi}^{-1}(t)$ for $t\neq 0$ with the corresponding element of the pencil in $X$. 
\begin{lemma}\label{lmToricClass}
The restriction of  the Poincar\'e dual of a class $[\Sigma_{t_1,\dots, t_m}]$ to a generic fiber of $\tilde{\pi}$ coincides with the restriction of the  Poincar\'e dual of the class $\sum_{i=1}^m t_iD_i$ to a generic fibre of $\pi$ in $H^2(X;\mathbb{R})$. Here we identify the generic fiber of $\tilde{\pi}$ with the generic fiber of $\pi$ via the blow down map. 
\end{lemma}

\begin{proof}
The claim follows if we prove it for each component $\tilde{\Sigma}_i$. Note that the restriction of the Poincar\'e dual of $\tilde{\Sigma}_i$ to the generic fiber $\tilde{\pi}^{-1}(t)$ represents the class $\PD[ \tilde{\Sigma}_i \cap \tilde{\pi}^{-1}(t)]$ where $\PD$ stands for Poincar\'e dual in the cohomology of the fibre $\tilde{\pi}^{-1}(t)$ and $\cap$ is the homological intersection. As the small blow up operation is local and is only performed on the degenerate fibre, we have that the homological intersection in the small blow up is the same as the intersection in the original manifold i.e $\PD[ \tilde{\Sigma}_i \cap \tilde{\pi}^{-1}(t)] = \PD[\Sigma_i \cap \pi^{-1}({t})]$. But $\PD[\Sigma_i \cap \pi^{-1}({t})] = \PD[D_i \cap \pi^{-1}(t)]$. Indeed, $\Sigma_i= B_i \times \CP^1$ where $B_i$ was the component of the base locus meeting $D_i$. 
\end{proof}

\begin{lemma}\label{lmArbitraryToricClass}
    Fix a toric symplectic form $\omega_0$ on $X$. Then we can find a nearby symplectic form $\omega$, an ordered $m$-tuple $\vec{t}$ and a Poincare dual $\alpha_{\vec{t}}$ to $\Sigma_{\vec{t}}$ such that $\tilde{\omega}_{\vec{t}}:=\pi^*\omega+\alpha_{\vec{t}}$ is a symplectic form on $E_{sb}$ taming $J$ and such that the induced form on a generic fiber is cohomologous to the one induced by restriction of $\omega_0.$ Moreover, we can make $\omega$ arbitarily close to $\omega_0$ and $\vec{t}$ arbitrarily small.
\end{lemma}
\begin{proof}
    Write $\omega_0=\omega_{\vec{\lambda}}=\sum \lambda_i\beta_i$ where $\beta_i$ is Poincar\'e dual to $D_i$. We want to show that we can pick $\vec{t}=(t_1\dots t_N)$ and a Poincar\'e dual $\alpha$ to $\Sigma_{\vec{t}}$ such that $\omega_1:=\tilde{\pi}^*\omega_{\vec{\lambda}-\vec{t}}+\alpha$ tames $J_{sb}$. 

    For this we revisit the construction in Lemma \ref{Lemma:Family_of_Form_small_blowup}. Fix open neighborhoods $U_j$ of the exceptional spheres $C_j$ which are biholomorphic to open neighborhoods of the zero section in $\mathcal{O}(-1)\oplus\mathcal{O}(-1)$, and let $\omega_{U_i}$ be a Kahler form on $\mathcal{O}(-1)\oplus\mathcal{O}(-1)$ for which the zero section has unit area. For each face $D_i$ let $\alpha_i$ be a Poincar\'e dual to $\tilde{\Sigma}_i$ which coincides with $\pm\omega_{U_j}$ on some smaller neighborhood $ V_j$ of $C_j$ whose closure is in $U_i$. The sign is determined by the intersection number. This can be assumed as shown in the proof of Lemma \ref{Lemma:Family_of_Form_small_blowup}.  

    There is an open neighborhood $W$ of $\vec{\lambda}$ so that $\omega_{\vec{\lambda'}}$ is symplectic for all $\lambda'\in W$. Shrinking $W$ slightly and fixing a background metric on $E_{sb}$ we can find an $\epsilon>0$ such that on $E_{sb}\setminus \cup V_j$, we have that if $v$ is any unit tangent vector then $\omega_{\lambda'}(v,Jv)>\epsilon$ for all $\lambda'\in W.$  Then for $\vec{t}$ small enough with respect to $\epsilon$ and such that $\vec{\lambda}-\vec{t}\in W$, we have that $\tilde{\pi}^*\omega_{\vec{\lambda}-\vec{t}}+\sum_{i}t_i\alpha_i$ tames $J$. By Lemma~\ref{lmToricClass}, we have the cohomology class of $\tilde{\pi}^*\omega_{\vec{\lambda}-\vec{t}}+\sum_{i}t_i\alpha_i$ is the same as $\om_0$ thus proving the claim. 
\end{proof}

\subsection{Holomorphic volume form on the resolved manifold}\label{Subsection:Holformonsmallblowup}
Consider the restriction $E_{sb}|_{\C}$ of the family $E_{sb}$ to the complement of $\infty$. That is, consider the small blow up of $E\cap (X\times\C)$. Let $\Omega_X$ be a meromorphic volume form on $X$ with poles of order $1$ along the toric boundary. Then the volume form $$\Omega_{tot}:=\frac{s_0d\lambda\wedge \Omega_X}{\lambda(s_0+\lambda s_1)}$$ has a pole of order $1$ along $E$. We then define $\Omega_E=res(\Omega_{tot})$ at the smooth points of $E$, where $res$ denotes the Poincar\'e residue. $\Omega_E$ is a non-vanishing form wherever it is defined. The pullback $\Omega_{E_{sb}}$ is defined and non-vanishing on the complement of the exceptional locus which is of codimension $2$. By Hartogs' property $\Omega_{E_{sb}}$ extends to $E_{sb}|_{\C}$ and is non-vanishing everywhere. Thus the restriction of $E_{sb}$ has trivial canonical bundle. Since the components of the central fiber are rational and the intersection of any two is either empty or a rational curve, the following Lemma follows readily. 

\begin{lemma} \label{E_{sb}isKulikov}
${E_{sb}}|_{\C}$ is a Kulikov model of type III.  
\end{lemma}
\qed
 
\subsection{Almost toric fibration on the small blow-up}
By Lemma~\ref{E_{sb}isKulikov} ${E_{sb}}|_{\C}$ is a Kulikov model of type III.  In this section, we show that it is in fact a \emph{symplectic} Kulikov model of type III by producing an almost toric fibration on the central fibre of the degeneration.
 To recap, we consider $s_0^{-1}(0) = \bigcup_i D_i$, where $D_i$ is a 4-dimensional toric variety whose toric boundary is precisely the divisor $\bigcup_{j\neq i} {D}_{ij}$ for ${D}_{ij}:= {D}_{i} \cap {D}_{j}$. 
We have constructed a family of symplectic forms $\tilde{\om}_t$ for appropriate $t=(t_1,t_2,\cdots,t_n)$ on the small blowup $E_{sb}$. We denote the new holomorphic fibration by  $\tilde{\pi}:E_{sb} \rightarrow \CP^1$ with $\tilde{\pi}^{-1}(0) = \bigcup \Tilde{D}_i$ where each $\Tilde{D}_i$ the proper transform under the small blow up operation of ${D}_i$ at all singular points on $E$. Write $\tilde{\omega}_i=\tilde{\omega}_t|_{\Tilde{D}_i}$. We now turn to construct an almost toric fibration on $(\Tilde{D}_i,\tilde{\omega}_i)$. Before doing that, we briefly recall the notion of an almost toric blow up. We refer the reader to Section 5.4 in \cite{Symington} and Section 9.1 in \cite{Evans_2023} for more details. 
\subsubsection{Almost toric blow-up}\label{Sec:ATFblowup} 

\begin{defn}[Almost toric blow-up]\label{Defn:ATFblowup}
Let $(M,\om)$ be a symplectic 4-manifold with an almost toric fibration $\mu: M \rightarrow \R^2$ such that the 1-stratum of
the image $B:=\mu(M)$ is non-empty. Then the almost toric blow up/nodal blow up gives us a symplectic manifold $(\tilde{M}, \tilde{\om})$ with an almost toric fibration whose base diagram is constructed as follows:
\begin{enumerate}
    \item Choose a base diagram
for $B$ such that for some $a, \epsilon > 0 $ the set $\{(p_1, p_2)~|~ p_1 > -\epsilon, p_2 \geq 0, p_1 + p_2 < a + \epsilon\}$ with
the boundary components marked by heavy lines represents a fibered neighborhood of a
point in the 1-stratum. Remove the triangle with vertices $(0, 0), (a, 0), (0, a)$. 
 \item Connect the two resulting vertices of the base with a pair of dotted lines (As in fig~\ref{fig:ATFblowup})
\end{enumerate}

\end{defn}
{
 \tikzset{every picture/.style={line width=0.75pt}} 
 \begin{figure}[H]
    \centering

\begin{tikzpicture}[scale=1]

\draw[thick] (0,0) -- (3,0) -- (0,3) -- cycle;

\draw[->, thick] (3.5,1.5) -- (6,1.5);
\node at (4.75,2.0) {blow-up};

\coordinate (A) at (6.5,0);
\coordinate (B) at (9.5,0);
\coordinate (C) at (6.5,3);

\draw[thick] (A) -- (7.75,0);
\draw[thick] (8.75,0) -- (B);

\draw[thick] (A) -- (C);
\draw[thick] (B) -- (C);

\coordinate (D) at (7.75,0);
\coordinate (E) at (8.75,0);
\coordinate (F) at (7.75,1);

\draw[dashed, thick] (D) -- (F);
\draw[dashed, thick] (E) -- (F);
\node[font=\Large] at (F) {$\times$};
\end{tikzpicture}
\caption{Almost toric blow-up }
\label{fig:ATFblowup}
 \end{figure}
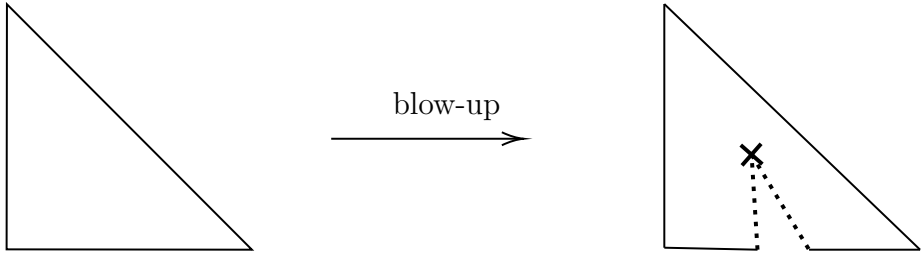
 }

\subsubsection{Construction of Almost toric fibration}

We recall the following useful claims.

\begin{thm}(Theorem 6.3 in \cite{Guillemin})\label{Theorem:Cohomologytoricform}
Let $(M,\om,\mu, \Delta)$ be a symplectic Toric manifold with image of momentum map $\mu: M \rightarrow \R^n$ being the Delzant polytope $\Delta$. Let $u_i$ denote the primitive normal vector to the boundary facets and let $\lambda_i$ denote the support constants defining the boundary hyperplanes. That is,   $\Delta$ is described by the  inequalities $\langle x,u_i\rangle \leq \lambda_i$ for $i \in \{1, \cdots, d\}$ where $d$ is the number of boundary facets. In particular, the facet $F_i$ is given by the relation $\langle x,u_i\rangle = \lambda_i$.  Let $C_i=\mu^{-1}(F_i)$ and $[C_i]$ their cohomology classes. Then the cohomology class of $\om$ is given by 
\begin{equation*}
    \frac{1}{2 \pi} [\om] = -\sum_{i=1}^d \lambda_i [C_i]
\end{equation*}
\end{thm}

\begin{thm}(\cite{Salamon})\label{Thm:Cohomologousformsarediffeo}
   Let $M= \CP^2\# k\overline{\CP^2}$ or $S^2 \times S^2$, if $\om$ and $\om'$ are cohomologous symplectic forms on $M$, then there is a diffeomorphism $\varphi$ that acts trivially on homology such that $\varphi^*\om= \om'$. 
\end{thm}

\begin{lemma}\label{lemma:Toricopenness}

Consider the symplectic toric manifold  $( M_{\Delta},\omega,\mu,\Delta)$ and let $C_i$ be the boundary of the momentum map as in  Theorem \ref{Theorem:Cohomologytoricform}. Then there exists a small neighbourhood $U \subset H^2(M_{\Delta},\R)$ of $[\om]$ such that every class $A \in U$ has a symplectic toric representative $\om_A$ whose toric boundary is given by $\bigcup_i C_i$.

\end{lemma}

\begin{proof}

By Theorem~\ref{Theorem:Cohomologytoricform}, the cohomology class of $\om$ is given by $ \frac{1}{2 \pi} [\om] = -\sum_{i=1}^d \lambda_i C_i$ for support constants $\lambda_i$. Perturb the cohomology class by perturbing the support constants $\lambda_i$ defining the facets $F_i$ slightly by $h_i[C_i]$ ($h_i \in \R$). For small enough perturbations of $\lambda_i$ the  resulting polytope still satisfies the Delzant conditions. Call this new perturbed polytope $\Delta'$ with boundary divisors $C_i'$ and the symplectic manifold associated to it by $(M_{\Delta'},\omega_{\Delta'})$. The fan structures associated to $\Delta$ and $\Delta'$ respectively are isomorphic and hence the complex projective varieties $X_{\Delta}$ and $X_{\Delta'}$ associated to $\Delta$ and $\Delta'$  respectively are equal i.e  $X_{\Delta} =X_{\Delta'}$ and have $\bigcup_i C_i$ as their torus invariant codimension 1 divisors.  We also have the following equivariant orientation preserving diffeomorphisms
$\phi_{\Delta}: M_{\Delta}\rightarrow X_{\Delta},$ and $\phi_{\Delta'}: M_{\Delta'} \rightarrow (X_{\Delta'}= X_{\Delta})$ that send the toric strata in $M_{\Delta}$ and $M_{\Delta'}$ to the toric strata in $X_{\Delta}$. Thus  $\om':= (\phi_{\Delta}^{-1})^* \circ \phi_{\Delta'}^*\om_{\Delta'}$ is a toric symplectic form (for the action pulled back from $M_{\Delta'}$ via $\phi_{\Delta'}^{-1}\circ \phi_{\Delta}$) in the cohomology class $[\om]+ h_i[C_i]$ and the boundary toric divisor for the pulled back action is $\phi_{\Delta}^{-1}\circ \phi_{\Delta'}(C'_i)$ which is equal to $C_i$ as required. 
By suitable perturbations of the support constants defining $C_i$ we can obtain such a toric symplectic form in any class in a small enough neighbourhood of $[\om]$.  
    
\end{proof}
    
\begin{rmk}
Note that Lemma 4.14 can also be proved more easily by performing a torus equivariant symplectic inflation (\cite[Corollary 1.4]{CJP} along the divisors $C_i$ to achieve nearby cohomology classes in a small neighbourhood $U$.
\end{rmk}
We turn to construct an almost toric fibration on $(\tilde{D}_i,\tilde{\om}_i,\cup \tilde{D}_{ij})$. 

\begin{lemma}\label{Lemma:ConstructionofATF}
Abbreviate $t=\vec{t}$. Let $\alpha_t$ and $\omega$ be as in Lemma \ref{lmArbitraryToricClass}, and let $\tilde{\omega}_t:=\pi^*\omega+\alpha_t$.
There is an $\epsilon>0$ such that for $|t|<\epsilon$ the symplectic manifold $(\tilde{D}_i, \tilde{\omega}_t|_{\tilde{D}_{i}})$ admits an almost toric fibration $\tilde{\mu}^i_t: (\tilde{D}_i,\tilde{\omega}_t|_{\tilde{D}_i}) \rightarrow \tilde{B}_i$ with base $\tilde{B}_i$ such that 
$(\mu^i_t)^{-1}(\partial {\tilde{B}}_i) = \bigcup_j \tilde{D}_{ij}$. Moreover, for fixed $\omega_0$, $\epsilon$ can be taken fixed for some open neighborhood $U$ of $\omega_0$
in the space of symplectic forms.
\end{lemma}

\begin{proof}
Let $\{E^i_1,\cdots, E^i_k\}$ denote the exceptional spheres on $\tilde{D}_i$ coming from performing the small blow-up (in the complex category) and let $s_j^t=\tilde{\omega}_t(E^i_j)$. By construction, for a given $j$, the size of the exceptional curves that come from blowing up points on $D_{ij}$ are all the same. We may choose for each $i$ a symplectic form on $D_i$ with the property that $\omega_t(D_{ij}) = \omega(D_{ij}) + l_js_j^t$ for each boundary divisor  $D_{ij}$ of $D_i$. Here $l_j$ is the number of blown up points on $D_{ij}$. Furthermore, each $D_{ij}$ is still symplectic for $\om_t$.  To see this, note the symplectic blow down is diffeomorphic to the complex blow down which reproduces $D$. This is realized by identifying a neighborhood of each exceptional curve with the tautological disc bundle. $\omega_t$ can be taken as the the pullback of the symplectic form  on the blow-down via this diffeomorphism.  

We proceed to construct a toric structure on $D_i$. By making $t$ small $\omega_t$ is arbitrarily close to $\omega$. By Lemma~\ref{lemma:Toricopenness}, there exists a form $\Omega_t$ cohomologous to $\om_t$ such that $\Omega_t$ admits a toric action with toric boundary $\cup_j D_{ij}$. Any symplectic toric manifold in dimension 4 is diffeomorphic to either $\CP^2\# k\overline{\CP^2}$ for some $k \in \mathbb{N}_{>0}$ or $S^2 \times S^2$. Thus by Theorem~\ref{Thm:Cohomologousformsarediffeo}, there exists a diffeomorphism $\varphi: (D_i,\om_t) \to (D_i,\Omega_t)$ that acts trivially on homology. Pulling back the toric action  on $(D_i,\Om_t)$ via $\varphi$, gives us a toric action on $(D_i,\om_t)$ with toric boundary $\varphi^{-1}(\cup_j D_{ij})$. As $\cup_j D_{ij}$ is symplectic for $\om_t$ by construction and $\varphi$ acts trivially on homology, we have that $\cup_j D_{ij}$ is in fact log Calabi-Yau for the symplectic form $\om_t$. 
Thus we can compose $\varphi$ with a symplectomorphism coming from the Torelli theorem (Theorem \ref{Theorem:Isotopyofdivisors}), to obtain a toric structure for $\om_t$ with $\cup_j D_{ij}$ as the toric boundary. 

We now proceed to construct an almost toric fibration on $(\tilde{D}_i,\tilde{\omega}_t|_{\tilde{D}_i})$. Let $(D'_i,\omega'_t)$ be the almost toric blow up of $(D_i,\omega_t)$ obtained by removing triangles of capacity $s^t_j$ from the moment polytope of $(D_i,\omega_t)$. Note that we use the fact that $s_j^t$ is small to guarantee the existence of disjoint triangles to perform the almost toric blow up. We can thus endow $D'_i$  with an almost toric fibration $\mu'_i:D'_i\to \tilde{B}_i$. Let $\bigcup_{j}D'_{ij}$ denote the boundary of this almost toric fibration. Since $\tilde{D}_i$ obtained by blowing up $i_j$ points on $D_{ij}$ in the complex category we obtain a diffeomorphism  $\tilde{D}_i\to D'_i$ which maps $\tilde{D}_{ij}$ to $D'_{ij}$. Moreover, it maps the class of $\tilde{\omega}_t|_{\tilde{D}_i}$ to the class of  $\omega'_t$ by equality of the areas of the exceptional spheres and the boundary spheres.

Furthermore, by Remark~\ref{Remark:Modifiedform_SymplecticOrthogonal} we know that $\tilde{D}_{ij}$ is symplectically orthogonal with respect to $\tilde{\omega}_t|_{D_i}$ and $\phi^{-1}(D'_{ij})$ is symplectically orthogonal with respect to $\phi^*\omega'_t$ by construction of the almost toric blow up. Applying the Torelli Theorem~\ref{Theorem:Isotopyofdivisors}, we obtain a symplectomorphism  $\psi:(\tilde{D}_i,\tilde{\omega}_t,\cup_j\tilde{D}_{ij})\to (D_i',\omega'_t ,\cup_jD'_{ij})$. 

Pulling back the almost toric fibration on $D'_i$ via $\psi$ gives us the required almost toric fibration on $(\tilde{D}_i, \tilde{\omega}_t|_{\tilde{D}_{i}})$ with boundary divisors $\cup_j \tilde{D}_{ij}$ as required.
 
\end{proof}
    
The following corollaries follow from Lemma~\ref{Lemma:Family_of_Form_small_blowup}, \ref{Lemma:ConstructionofATF} and Lemma \ref{E_{sb}isKulikov}.
\begin{cor}\label{CYToriKulikov}
    For $|\vec{t}|$ small enough, $(E_{sb},\tilde{\omega}_{\vec{t}})$ is a symplectic Kulikov model of type III for the degeneration defined by the resolution of the pencil associated with the anti-canonical sections $s_0,s_1$. 
\end{cor}
\begin{cor}
Let $(X,\omega,\mu_X)$ be a toric Fano 3-fold with momentum polytope $P$ and let $\pi: E\rightarrow \C$ be the pencil of as defined in Equation \ref{Equation:Pencilofsections}.  For $|\vec{t}|$ small enough, the general fiber of the resolved fibration $\tilde{\pi}$ endowed with the restriction of $\tilde{\omega}_{\vec{t}}$ carries an almost toric fibration whose base can be identified with the boundary of the moment polytope $\partial P$.
\end{cor}
 
\begin{rmk}
  One can also appeal to Theorem 2 in \cite{Li-Min-Ning}  to directly produce the almost toric fibration  $\mu: (\tilde{D}_i, \tilde{\omega}_{\vec{t}}|_{D_{i}}) \rightarrow B$ with boundary $\bigcup_j D_{ij}$. However, we chose to give a self contained proof to have an explicit description of the almost toric fibration. We use the construction in Lemma~\ref{Lemma:ConstructionofATF} to provide an explicit isomorphism with the Gross-Siebert construction in Section~\ref{SecGrossSiebert}.
\end{rmk} 

\subsection{Relation to the Gross-Siebert construction}\label{SecGrossSiebert}
Let $\cX_{t}=(E_{sb},\tilde{\omega_t})$ be a symplectic Kulikov model as in the previous subsection. 
We now turn to give an explicit description of the nodal integral affine structure produced by Theorem \ref{tmNAffine} on the intersection complex $\Delta_{\cX_t}$.  

Let $N\simeq\Z^3$ be a three dimensional lattice and let  $P\subset N_{\mathbb{R}}=N\otimes \R$ be a Delzant polytope cut out by half spaces $\langle \nu_i,x\rangle\leq b_i$ where $\nu_i$ are primitive generators in the dual lattice $M=\Hom(N,\Z)$ and $b_i$ are positive real numbers. Then $\partial P$ carries a piecewise integral affine structure induced from the embedding in $N_{\mathbb{R}}$. By the Delzant condition, we can extend this integral affine structure to include open neighborhoods of the vertices by taking the canonical structure introduced in Definition \ref{dfCanIntAff}.  Let $S=\{p_1,\dots, p_N\}$ be the set of midpoints of the edges of $\partial P$. Then this extends uniquely to an integral affine structure on $\partial P\setminus S$. It is not hard to see that the monodromy around a point $p\in S$ is a shear whose eigenline is the edge containing $p$. We call this nodal integral affine structure the \emph{Gross-Siebert} structure and denote it by $\mathcal{A}_{GS}$. 

\begin{rmk}
  This nodal integral affine structure is constructed 
  in \cite{GrossSiebert2002}  on the intersection complex of the original toric  degeneration $\pi:X\to \CP^1$. It works more generally without the Delzant assumption. The polytope $P$ is taken in \cite{GrossSiebert2002}  to be the Newton polytope of an ample line bundle on $X$. $P$ is then a lattice polytope. This is extended to an integral affine structure on the complement of the centers of the edges by specifying a \emph{fan structure} for each vertex $v$ which is readily seen to induce the canonical structure from definition \ref{dfCanIntAff}. In our setting, rather than considering the Newton polytope, we consider for each toric form $\omega$ on $X$ the moment polytope $P_{\omega}$.   
\end{rmk}
 
We then have the following claim:
\begin{prop}\label{prpGSIAStr}
    Assume the moment polytope $P_{\omega}$ of $X$ is cut out by the inequalities $\langle\nu_i,x\rangle\leq b_i$ for $1\leq i\leq m$. Suppose we carry out the greedy small blow up with parameters $t_1> t_2\dots>t_m$. Then after nodal slides along monodromy invariant lines the induced nodal integral affine structure $\mathcal{A}_{\vec{t}}$ on the intersection complex is integral affine isomorphic to the Gross-Siebert integral affine structure on the polytope $P_{\omega_{\vec{t}}}$ cut out by $\langle \nu_i,x\rangle\leq b_i+t_i$.
\end{prop}

Before proving Proposition \ref{prpGSIAStr}  we need to consider a version of Proposition \ref{propPcwsNAffine} and Theorem \ref{tmNAffine} applied to the intersection complex for a hybrid degeneration.
That is, consider a degeneration arising by starting with a smooth toric Fano and carrying out some of the steps of greedy small blowup.
Call this a \emph{hybrid degeneration}. All the triple intersections are normal crossings, and all the components are equipped with almost toric fibrations. But there are some double point singularities in the total space. Moreover, any face which contains a double point is a toric variety by greediness. To such a degeneration we can associate an intersection complex $\Delta$. Moreover, the construction in the proof \ref{propPcwsNAffine} can be applied to endow $\Delta$ with an integral affine structure up to isomorphism. We then have the following lemma. 

\begin{lemma}\label{lmNAffine}
    Suppose $F$ is a face of $\Delta$ corresponding to a component of the central fiber which contains no double point. Then there is a neighborhood $V$ of $F$ on which there is a unique extension of the  piecewise nodal integral affine structure on $V$ to a strict nodal integral affine structure which agrees with the canonical  structure of Definition \ref{dfCanIntAff} on neighborhoods of the vertices.
\end{lemma}
\begin{proof}
    This follows as in Theorem~\ref{tmNAffine} but done for each face of the hybrid degeneration. 
\end{proof}

We will also need the following definition. 

\begin{defn}{[Developing map (Remark 2.20 in \cite{Evans_2023})]}
Let $M$ be an $n$-dimensional integral affine manifold. Then there is a globally defined local diffeomorphism $D: \tilde{M} \rightarrow \R^n$ from the universal cover $\Tilde{M}$, into $\R^n$ such that the integral affine structure of $\Tilde{M}$ (coming from the covering map to M) agrees with the pullback of the integral
affine structure along $D$. This is called the developing map. 
\end{defn}

\begin{proof}[Proof of Proposition \ref{prpGSIAStr}]
Order the components of $\partial P$ as $F_1,F_2\dots F_m$. For $k=1,\dots m$ let $\Delta^{(k)}$ be the boundary $\partial P_{\omega}$ of the moment polytope endowed with the piecewise nodal integral affine structure obtained by Proposition \ref{lmNAffine} from the hybrid degeneration associated with carrying out the first $k$ steps of the greedy blow-up. Note $\Delta^{(k)}$ is obtained from $\Delta^{(k-1)}$ by putting on $F_k$ the nodal integral affine structure whose edges have the same affine length as that of the toric integral affine structure with nodes placed an affine distance of $t_k$ from the boundary. Thus, it attaches continuously to $\Delta^{(k-1)}\setminus F_k.$ Theorem \ref{lmNAffine} then tells us the piecewise integral affine structure extends smoothly across $\partial F_k$. 
\tikzset{every picture/.style={line width=0.75pt}} 
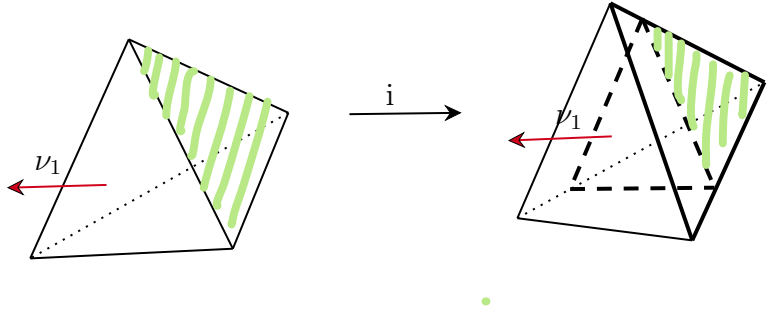
\begin{figure}[H]
    \centering
    \caption{The inclusion $Q^{0}\setminus\partial_1Q^{(0)}\to Q^{(1)}\setminus \partial_1Q^{(1)}$ }
    \label{fig:inclusion}
\begin{tikzpicture}[x=0.75pt,y=0.75pt,yscale=-0.7,xscale=0.7] 
[id:da038881653948903194] \draw [line width=1.5] [dash pattern={on 5.63pt off 4.5pt}] (480,43) -- (532,165) ; 
[id:da46930993744391714] \draw [line width=1.5] [dash pattern={on 5.63pt off 4.5pt}] (480,43) -- (428,166) ; 
[id:da05775791796222274] \draw [line width=1.5] [dash pattern={on 5.63pt off 4.5pt}] (428,166) -- (532,165) ; 
[id:da15476197947590087] \draw [dash pattern={on 0.84pt off 2.51pt}] (390,187) -- (568,89) ; 
[id:da31190267816454686] \draw [color={rgb, 255:red, 0; green, 0; blue, 0 } ,draw opacity=1 ][line width=1.5] (457,32) -- (568,89) ; 
[id:da4930129637659202] \draw [color={rgb, 255:red, 0; green, 0; blue, 0 } ,draw opacity=1 ][line width=1.5] (568,89) -- (516,203) ; 
[id:da15288893086629518] \draw (457,32) -- (390,187) ; 
[id:da8565839864457827] \draw [color={rgb, 255:red, 0; green, 0; blue, 0 } ,draw opacity=1 ][line width=1.5] (457,32) -- (516,203) ; 
[id:da8086646223680668] \draw (390,187) -- (516,203) ; 
[id:da1777731909884952] \draw (269,112) -- (346,111.04) ; \draw [shift={(349,111)}, rotate = 179.28] [fill={rgb, 255:red, 0; green, 0; blue, 0 } ][line width=0.08] [draw opacity=0] (10.72,-5.15) -- (0,0) -- (10.72,5.15) -- (7.12,0) -- cycle ; 
[id:dp34757533971838406] \draw [color={rgb, 255:red, 184; green, 233; blue, 134 } ,draw opacity=1 ][line width=3] [line join = round][line cap = round] (491,52) .. controls (491,56) and (491,60) .. (491,64) ; 
[id:dp2618366541540159] \draw [color={rgb, 255:red, 184; green, 233; blue, 134 } ,draw opacity=1 ][line width=3] [line join = round][line cap = round] (499,55) .. controls (499,64.67) and (499,74.66) .. (499,84) ; 
[id:dp45763217056494643] \draw [color={rgb, 255:red, 184; green, 233; blue, 134 } ,draw opacity=1 ][line width=3] [line join = round][line cap = round] (508,61) .. controls (508,73.12) and (504.29,86.01) .. (506,98) .. controls (506.21,99.48) and (508,100.51) .. (508,102) ; 
[id:dp48521254287420745] \draw [color={rgb, 255:red, 184; green, 233; blue, 134 } ,draw opacity=1 ][line width=3] [line join = round][line cap = round] (519,69) .. controls (519,67.67) and (519.16,71.68) .. (519,73) .. controls (518.53,77.03) and (517.4,80.96) .. (517,85) .. controls (516.05,94.55) and (513.09,113.17) .. (517,121) ; 
[id:dp7058835398383511] \draw [color={rgb, 255:red, 184; green, 233; blue, 134 } ,draw opacity=1 ][line width=3] [line join = round][line cap = round] (531,76) .. controls (531,92.02) and (527.06,110.7) .. (526,127) .. controls (525.57,133.65) and (526,153.67) .. (526,147) ; 
[id:dp8051431787247718] \draw [color={rgb, 255:red, 184; green, 233; blue, 134 } ,draw opacity=1 ][line width=3] [line join = round][line cap = round] (543,84) .. controls (538.72,96.84) and (538,120.13) .. (538,133) ; 
[id:dp2639732298441292] \draw [color={rgb, 255:red, 184; green, 233; blue, 134 } ,draw opacity=1 ][line width=3] [line join = round][line cap = round] (554,84) .. controls (553.87,85.59) and (552,117.66) .. (552,112) ; 
[id:da3307700442287226] \draw [color={rgb, 255:red, 208; green, 2; blue, 27 } ,draw opacity=1 ] (458,128) -- (387,130.88) ; \draw [shift={(384,131)}, rotate = 357.68] [fill={rgb, 255:red, 208; green, 2; blue, 27 } ,fill opacity=1 ][line width=0.08] [draw opacity=0] (10.72,-5.15) -- (0,0) -- (10.72,5.15) -- (7.12,0) -- cycle ; 
[id:da8833101741049766] \draw (110,58) -- (185,209) ; 
[id:da6255660865381705] \draw (110,58) -- (39,216) ; 
[id:da22738745661456383] \draw (39,216) -- (185,209) ; 
[id:da2527114373230739] \draw [dash pattern={on 0.84pt off 2.51pt}] (225,111) -- (39,216) ; 
[id:da873277983442696] \draw (110,58) -- (225,111) ; 
[id:da545503884980962] \draw (225,111) -- (185,209) ; 
[id:da6006178501191979] \draw [color={rgb, 255:red, 208; green, 2; blue, 27 } ,draw opacity=1 ] (94,163) -- (26,164.92) ; \draw [shift={(23,165)}, rotate = 358.39] [fill={rgb, 255:red, 208; green, 2; blue, 27 } ,fill opacity=1 ][line width=0.08] [draw opacity=0] (10.72,-5.15) -- (0,0) -- (10.72,5.15) -- (7.12,0) -- cycle ; 
[id:dp2560059326465801] \draw [color={rgb, 255:red, 184; green, 233; blue, 134 } ,draw opacity=1 ][line width=3] [line join = round][line cap = round] (367,247) .. controls (367,247) and (367,247) .. (367,247) ; 
[id:dp7449356659913429] \draw [color={rgb, 255:red, 184; green, 233; blue, 134 } ,draw opacity=1 ][line width=3] [line join = round][line cap = round] (124,66) .. controls (124,69.82) and (122.56,77.87) .. (121,81) ; 
[id:dp26300122894456746] \draw [color={rgb, 255:red, 184; green, 233; blue, 134 } ,draw opacity=1 ][line width=3] [line join = round][line cap = round] (132,71) .. controls (132,75.59) and (123.77,103.23) .. (129,98) ; 
[id:dp3604941179019505] \draw [color={rgb, 255:red, 184; green, 233; blue, 134 } ,draw opacity=1 ][line width=3] [line join = round][line cap = round] (144,74) .. controls (144,83.76) and (137,99.84) .. (137,113) ; 
[id:dp0710345106432606] \draw [color={rgb, 255:red, 184; green, 233; blue, 134 } ,draw opacity=1 ][line width=3] [line join = round][line cap = round] (157,81) .. controls (150.17,81) and (151.47,102.14) .. (150,109) .. controls (149.8,109.92) and (144.85,127.15) .. (148,124) ; 
[id:dp03070458640519791] \draw [color={rgb, 255:red, 184; green, 233; blue, 134 } ,draw opacity=1 ][line width=3] [line join = round][line cap = round] (169,87) .. controls (169,98.6) and (161.75,112.55) .. (159,124) .. controls (157.66,129.6) and (157.19,135.37) .. (156,141) .. controls (155.9,141.46) and (155,142.47) .. (155,142) ; 
[id:dp835468017140248] \draw [color={rgb, 255:red, 184; green, 233; blue, 134 } ,draw opacity=1 ][line width=3] [line join = round][line cap = round] (183,95) .. controls (179.58,118.97) and (169.45,141.4) .. (164,165) ; 
[id:dp7557175244497981] \draw [color={rgb, 255:red, 184; green, 233; blue, 134 } ,draw opacity=1 ][line width=3] [line join = round][line cap = round] (197,99) .. controls (191.56,119.67) and (185.33,140.64) .. (179,161) .. controls (177.62,165.43) and (175.47,169.6) .. (174,174) .. controls (173.58,175.26) and (173.06,178.94) .. (174,178) ; 
[id:dp723174229761047] \draw [color={rgb, 255:red, 184; green, 233; blue, 134 } ,draw opacity=1 ][line width=3] [line join = round][line cap = round] (210,103) .. controls (205.36,123.1) and (198.95,142.37) .. (193,162) .. controls (190.86,169.06) and (188.1,175.93) .. (186,183) .. controls (185.22,185.64) and (184,193.75) .. (184,191) ; 
\draw (293,88) node [anchor=north west][inner sep=0.75pt] [align=left] {i}; 
\draw (415,107) node [anchor=north west][inner sep=0.75pt] [align=left] {$\nu_1$}; 
\draw (40,140) node [anchor=north west][inner sep=0.75pt] [align=left] {$\nu_1$}; 
\end{tikzpicture}
\end{figure}
Let $Q^{(k)}$ be the boundary of the polytope obtained from $P_{\omega}$ by replacing the first $k$ defining inequalities $\langle \nu_i,x\rangle\leq b_i$ with $\langle \nu_i,x\rangle\leq b_i+t_i$. We inductively construct piecewise integral affine isomorphisms $$f_k:Q^{(k)}\to \Delta^{(k)}$$
such that the following hold:
\begin{enumerate}
    \item The restriction of $f_k$ to any face is a strict integral affine isomorphism. 
    \item $f_k$ agrees with the restriction of $f_{k-1}$ to $Q^{(k-1)}\setminus \partial_k Q^{(k-1)}$. Here $\partial_k Q^{(k-1)}$ is the face normal to $\nu_k$. For $k=1$ we take $f_0:Q^{(0)}\to \Delta^{(0)}$ to be the identity.
    \item The image under $f_k$ of any edge occurring in the first $k$ faces of $Q^{(k)}$ contains at least one node.
\end{enumerate}

The claim then follows by putting $k=m$. Indeed, $f_m$ is a piecewise integral affine isomorphism from $\partial P_{\omega_{\vec{t}}}$ with the affine structure induced from $N_{\R}$ to $\Delta_{\cX_{\vec{t}}}$ with the nodal integral affine structure $\mathcal{A}_{\vec{t}}$. Moreover, $f_m$ is strictly integral affine on the faces of $Q^{(m)}=\partial P_{\omega_{\vec{t}}}$. Since the integral affine structure in the neighborhood of each vertex of $\Delta$ is the canonical one, we have $\mathcal{A}_{\vec{t}}=\mathcal{A}_{GS}$.

\tikzset{every picture/.style={line width=0.75pt}} 
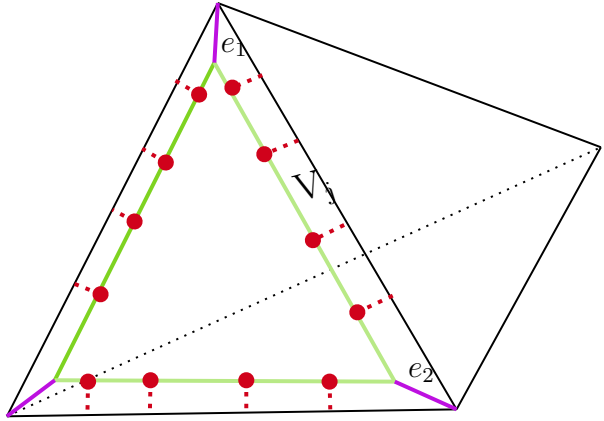
\begin{figure}[H]
    \centering
    \begin{tikzpicture}[x=0.75pt,y=0.75pt,yscale=-0.7,xscale=0.7] 
[id:da9218287184638247] \draw (300,15) -- (148,313) ; 
[id:da01717133214506572] \draw (300,15) -- (472,308) ; 
[id:da42684115616129414] \draw (148,313) -- (472,308) ; 
[id:da08314832122389437] \draw [dash pattern={on 0.84pt off 2.51pt}] (148,313) -- (576,119) ; 
[id:da45851105823226146] \draw (300,15) -- (576,119) ; 
[id:da9625768229150794] \draw (576,119) -- (472,308) ; 
[id:da3783312263714429] \draw [color={rgb, 255:red, 126; green, 211; blue, 33 } ,draw opacity=1 ][line width=1.5] (297.5,58) -- (182.5,287) ; 
[id:da8642434179856998] \draw [color={rgb, 255:red, 184; green, 233; blue, 134 } ,draw opacity=1 ][line width=1.5] (182.5,287) -- (427.5,288) ; 
[id:da9367557507116399] \draw [color={rgb, 255:red, 184; green, 233; blue, 134 } ,draw opacity=1 ][line width=1.5] (297.5,58) -- (427.5,288) ; 
[id:da5344566371506331] \draw [color={rgb, 255:red, 189; green, 16; blue, 224 } ,draw opacity=1 ][line width=1.5] (300,15) -- (297.5,58) ; 
[id:da3890167252979082] \draw [color={rgb, 255:red, 189; green, 16; blue, 224 } ,draw opacity=1 ][line width=1.5] (148,313) -- (182.5,287) ; 
[id:da7686883960512433] \draw [color={rgb, 255:red, 189; green, 16; blue, 224 } ,draw opacity=1 ][line width=1.5] (427.5,288) -- (472,308) ; 
[id:da6012898051223005] \draw [color={rgb, 255:red, 208; green, 2; blue, 27 } ,draw opacity=1 ][line width=1.5] [dash pattern={on 1.69pt off 2.76pt}] (262.5,130) -- (245.5,120) ; \draw [shift={(262.5,130)}, rotate = 210.47] [color={rgb, 255:red, 208; green, 2; blue, 27 } ,draw opacity=1 ][fill={rgb, 255:red, 208; green, 2; blue, 27 } ,fill opacity=1 ][line width=1.5] (0, 0) circle [x radius= 4.36, y radius= 4.36] ; 
[id:da2745185438219655] \draw [color={rgb, 255:red, 208; green, 2; blue, 27 } ,draw opacity=1 ][line width=1.5] [dash pattern={on 1.69pt off 2.76pt}] (269.5,71) -- (286.5,81) ; \draw [shift={(286.5,81)}, rotate = 30.47] [color={rgb, 255:red, 208; green, 2; blue, 27 } ,draw opacity=1 ][fill={rgb, 255:red, 208; green, 2; blue, 27 } ,fill opacity=1 ][line width=1.5] (0, 0) circle [x radius= 4.36, y radius= 4.36] ; 
[id:da9944451519817689] \draw [color={rgb, 255:red, 208; green, 2; blue, 27 } ,draw opacity=1 ][line width=1.5] [dash pattern={on 1.69pt off 2.76pt}] (240,172.5) -- (224,164) ; \draw [shift={(240,172.5)}, rotate = 207.98] [color={rgb, 255:red, 208; green, 2; blue, 27 } ,draw opacity=1 ][fill={rgb, 255:red, 208; green, 2; blue, 27 } ,fill opacity=1 ][line width=1.5] (0, 0) circle [x radius= 4.36, y radius= 4.36] ; 
[id:da6495024226496728] \draw [color={rgb, 255:red, 208; green, 2; blue, 27 } ,draw opacity=1 ][line width=1.5] [dash pattern={on 1.69pt off 2.76pt}] (215.5,225) -- (195.5,217) ; \draw [shift={(215.5,225)}, rotate = 201.8] [color={rgb, 255:red, 208; green, 2; blue, 27 } ,draw opacity=1 ][fill={rgb, 255:red, 208; green, 2; blue, 27 } ,fill opacity=1 ][line width=1.5] (0, 0) circle [x radius= 4.36, y radius= 4.36] ; 
[id:da8105087479207714] \draw [color={rgb, 255:red, 208; green, 2; blue, 27 } ,draw opacity=1 ][line width=1.5] [dash pattern={on 1.69pt off 2.76pt}] (206.5,288) -- (206,313) ; \draw [shift={(206.5,288)}, rotate = 91.15] [color={rgb, 255:red, 208; green, 2; blue, 27 } ,draw opacity=1 ][fill={rgb, 255:red, 208; green, 2; blue, 27 } ,fill opacity=1 ][line width=1.5] (0, 0) circle [x radius= 4.36, y radius= 4.36] ; 
[id:da8667704012615938] \draw [color={rgb, 255:red, 208; green, 2; blue, 27 } ,draw opacity=1 ][line width=1.5] [dash pattern={on 1.69pt off 2.76pt}] (251.5,287) -- (251,312) ; \draw [shift={(251.5,287)}, rotate = 91.15] [color={rgb, 255:red, 208; green, 2; blue, 27 } ,draw opacity=1 ][fill={rgb, 255:red, 208; green, 2; blue, 27 } ,fill opacity=1 ][line width=1.5] (0, 0) circle [x radius= 4.36, y radius= 4.36] ; 
[id:da8082658872577388] \draw [color={rgb, 255:red, 208; green, 2; blue, 27 } ,draw opacity=1 ][line width=1.5] [dash pattern={on 1.69pt off 2.76pt}] (320.5,287) -- (320.5,309.5) ; \draw [shift={(320.5,287)}, rotate = 90] [color={rgb, 255:red, 208; green, 2; blue, 27 } ,draw opacity=1 ][fill={rgb, 255:red, 208; green, 2; blue, 27 } ,fill opacity=1 ][line width=1.5] (0, 0) circle [x radius= 4.36, y radius= 4.36] ; 
[id:da9588349820970687] \draw [color={rgb, 255:red, 208; green, 2; blue, 27 } ,draw opacity=1 ][line width=1.5] [dash pattern={on 1.69pt off 2.76pt}] (380.5,288) -- (381,308) ; \draw [shift={(380.5,288)}, rotate = 88.57] [color={rgb, 255:red, 208; green, 2; blue, 27 } ,draw opacity=1 ][fill={rgb, 255:red, 208; green, 2; blue, 27 } ,fill opacity=1 ][line width=1.5] (0, 0) circle [x radius= 4.36, y radius= 4.36] ; 
[id:da868071559790941] \draw [color={rgb, 255:red, 208; green, 2; blue, 27 } ,draw opacity=1 ][line width=1.5] [dash pattern={on 1.69pt off 2.76pt}] (332,67) -- (310.5,76) ; \draw [shift={(310.5,76)}, rotate = 157.29] [color={rgb, 255:red, 208; green, 2; blue, 27 } ,draw opacity=1 ][fill={rgb, 255:red, 208; green, 2; blue, 27 } ,fill opacity=1 ][line width=1.5] (0, 0) circle [x radius= 4.36, y radius= 4.36] ; 
[id:da3031956738409429] \draw [color={rgb, 255:red, 208; green, 2; blue, 27 } ,draw opacity=1 ][line width=1.5] [dash pattern={on 1.69pt off 2.76pt}] (358,114) -- (333.5,124) ; \draw [shift={(333.5,124)}, rotate = 157.8] [color={rgb, 255:red, 208; green, 2; blue, 27 } ,draw opacity=1 ][fill={rgb, 255:red, 208; green, 2; blue, 27 } ,fill opacity=1 ][line width=1.5] (0, 0) circle [x radius= 4.36, y radius= 4.36] ; 
[id:da21610131472906757] \draw [color={rgb, 255:red, 208; green, 2; blue, 27 } ,draw opacity=1 ][line width=1.5] [dash pattern={on 1.69pt off 2.76pt}] (391,175) -- (368.5,186) ; \draw [shift={(368.5,186)}, rotate = 153.95] [color={rgb, 255:red, 208; green, 2; blue, 27 } ,draw opacity=1 ][fill={rgb, 255:red, 208; green, 2; blue, 27 } ,fill opacity=1 ][line width=1.5] (0, 0) circle [x radius= 4.36, y radius= 4.36] ; 
[id:da7666526733403244] \draw [color={rgb, 255:red, 208; green, 2; blue, 27 } ,draw opacity=1 ][line width=1.5] [dash pattern={on 1.69pt off 2.76pt}] (426,226) -- (400.5,238) ; \draw [shift={(400.5,238)}, rotate = 154.8] [color={rgb, 255:red, 208; green, 2; blue, 27 } ,draw opacity=1 ][fill={rgb, 255:red, 208; green, 2; blue, 27 } ,fill opacity=1 ][line width=1.5] (0, 0) circle [x radius= 4.36, y radius= 4.36] ; 
\draw (300,38) node [anchor=north west][inner sep=0.75pt] [align=left] {$e_1$}; 
\draw (435,273) node [anchor=north west][inner sep=0.75pt] [align=left] {$e_2$};
Text Node \draw (348.25,139.45) node [anchor=north west][inner sep=0.75pt] [font=\large,rotate=-344.5,xslant=-0.29] [align=left] {$V_j$};
\end{tikzpicture}
     \caption{Hybrid degeneration after performing small blow up on the face $F_1$}
\label{fig:Hybriddegeneration}
\end{figure}

For the base case start by blowing up the double points mapping to edges of $F_1$ by an amount $t_1$. To construct the map $f_1$ we start with the obvious inclusion $i:Q^{(0)}\setminus\partial_1Q^{(0)}\hookrightarrow Q^{(1)}\setminus \partial_1Q^{(1)}$, illustrated in Figure~\ref{fig:inclusion}. We have the identification $f_0: Q^{(0)} \xrightarrow{=} \Delta^{(0)}$. So we take $f_1\circ i = f_0$.

For a $j\neq 1$ such that the face $\partial_jQ^{(1)}$ meets $\partial_1Q^{(1)}$ we extend $f_1$ to all of $\partial_jQ^{(1)}$ by the following procedure. Consider inside $F_1$ the polygon $p_1\subset F_1$ whose edges are monodromy invariant segments through the nodes (e.g., the polygon defined by the union of the green line segments in Figure~\ref{fig:Hybriddegeneration}). Let $V_j\subset F_1$ be the polygon cut out by the segment $\ell_{1,j}=\partial_jQ^{(0)}\cap F_1$, the edge $\ell_2$ of $p_1$ parallel to $\ell_{1,j}$ and the segments $e_1,e_2$ connecting the vertices of $\ell_{1,j}$ with those of $\ell_2$ (See Figure~\ref{fig:Hybriddegeneration}).

\tikzset{every picture/.style={line width=0.75pt}} 
\begin{figure}[H]
    \centering
    
    \caption{Construction of $f_1$}
    \label{fig:f_0}
\input{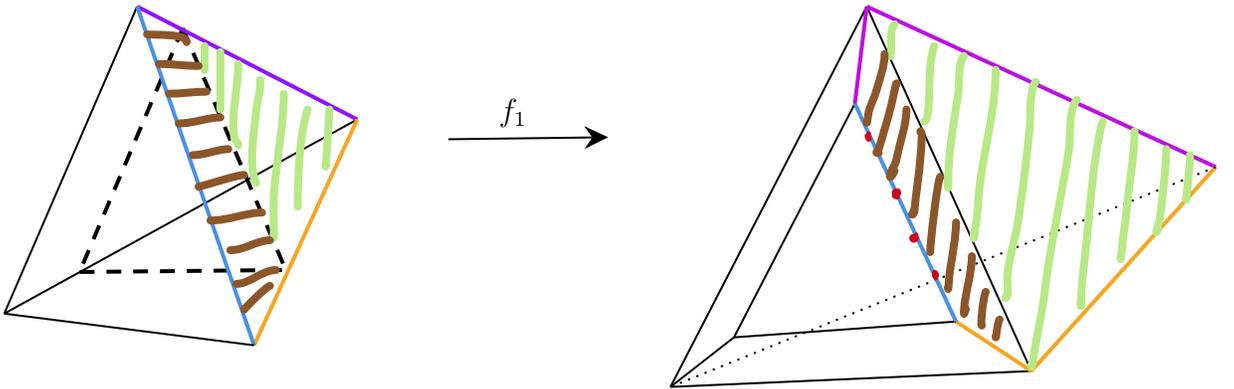}
\end{figure}
Consider the union $U_j= f_0(\partial_jQ^{(0)})\cup V_j\subset \Delta^{(k)}$. We show now that $f_1$ extends uniquely to an affine isomorphism $\partial_jQ^{(1)}\to U_j$.
By Lemma \ref{lmNAffine} the nodal piecewise integral affine structure on $U_j$  is integral affine with no nodes.  Since $U_j$ is simply connected there is a developing map $D_j:U_j \rightarrow \R^2$. Consider the composition map $D_j \circ f_1$ defined thus far on $i(\partial_jQ^{(0)})$ and extend to $\partial_jQ^{(1)}$ linearly.
Call this extended map $\widetilde{\left(D_j \circ f_0\right)}: \partial_j Q^{(1)} \rightarrow \R^2$. We claim the image of this map is exactly $D_j(U_j)$. Namely, we argue that the two images are bounded by the same lines in the affine space.  For this we first observe that, by Lemma~\ref{Lemma:Baselocusintersectsallstrata} that the base locus intersects each codimension 1 strata and hence after performing the small blow-up on each of the edges of $F_1$ in the hybrid degeneration,  each edge of $p_1$ is parallel to an edge in the boundary of $F_1$ and is an affine distance of $t_1$ away from it. Furthermore, by the description of the integral affine structure near the vertices as in Definition~\ref{dfCanIntAff}, it follows that the segments $e_1,e_2$ connect smoothly with edges of $F_j$ incident on the vertices $v_1,v_2$ of $\ell_{1,j}= F_j\cap F_1$. In other words, the union of $e_i$ with the edges of $F_j$ incident on $v_j$ is a straight line in this affine chart. Let us denote these two lines by $\tilde{e}_1$ and $\tilde{e}_2$ respectively. Now, $f_0$ maps the boundary edges of $\partial_jQ^{(0)}$ that transversely intersect $\partial_1Q^{(0)}$ into the lines $\tilde{e}_1$ and $\tilde{e}_2$. This is depicted by the images of the edges labelled purple and orange in Figure~\ref{fig:f_0}. The proof that $D_j(U_j) = \widetilde{\left(D_j \circ f_0\right)}(\partial_j Q^{(1)})$ is then complete if we show that $\widetilde{\left(D_j \circ f_0\right)}(\partial_1 Q^{(1)} \cap \partial_j Q^{(1)})$ maps to the affine line $D_j (\ell_2)$ (denoted by the blue edges in Figure~\ref{fig:f_0}). We prove this by showing that the above two affine lines are the level sets of an affine function at the same level.  
Consider the following two affine functions:
\begin{align*}
    H_1: \partial_j Q^{(1)} &\longrightarrow \R \\
    x &\mapsto \langle x,\nu_1\rangle,
\end{align*}
and,
\begin{align*}
   G_1:=\begin{cases}
    \langle x,(f_0)_*\nu_1\rangle \text{~~if $x\in f_0(\partial_jQ^{(0)})$} \\
\langle x,\vec{\ell}_{1,j}\rangle \text{~~if $x \in V_j$}, 
\end{cases} 
\end{align*}
where $\vec{\ell}_{1,j}$ denotes the primitive integral normal to the line $\ell_{1,j}$. We take the direction of $\vec{\ell}_{1,j}$ to be in the same direction as primitive normal to the level sets of $\langle \cdot,(f_0)_*\nu_1\rangle$. Note that the line $\partial_1 Q^{(1)} \cap \partial_j Q^{(1)}$ is defined by $H_1^{-1}(b_1+t_1)$. With these choices, it follows that $G_1$ is an integral affine function on 
$U_j$
having the monodromy invariant line $\ell_2$ as a level set with value $b_1 + t_1$. Pushing forward the maps $H_1$ and $G_1$ to $\R^2$ by $\widetilde{\left(D_j \circ f_0\right)}$ and $D_j$ respectively, we see that the pushforward maps agree on the overlap $D_j(f_0(\partial_jQ^{(0)})) = \widetilde{\left(D_j \circ f_0\right)}(\partial_jQ^{(0)})$. As the two maps are integral affine maps on $\R^2$ that agree on an open set, they agree everywhere and hence their level sets at $b_1+t_1$ is the same set. This completes the proof that $\widetilde{\left(D_j \circ f_0\right)}(\partial_1 Q^{(1)} \cap \partial_j Q^{(1)})$ maps to the affine line $D_j (\ell_2)$.  
 
It follows that after blowing up $F_1$, the integral affine structure in a neighborhood of $F_1$, is the required one and the map $f_1$ is defined as follows: $f_1(x) := (D_j)^{-1}\circ \widetilde{\left(D_j \circ f_0\right)}(x)$.

Repeating this for all $j\neq 1$ such that the face $\partial_jQ^{(1)}$ meets $\partial_1Q^{(1)}$, this extends $f_1$ to all of $Q^(1)\setminus\partial_1Q^{(1)}$. But then $\partial_1Q^{(1)}$ maps onto $p_1$ in the obvious way. The map $f_1$ thus defined satisfies the first two inductive assumptions by construction. It satisfies the third since the edges in $\partial_1Q^{(1)}$ map to the edges of $p_1$ each of which contains a node as already noted.

Suppose inductively we have proven the claim for $k-1$. To deduce it for $k$ observe the face $F_k$ is bordered by segments which either have singularities that have not yet been resolved, or which contain nodes. The effect of resolving the singularities is to push all the nodes into the interior of $F_k$. Indeed, for the new nodes this is by definition and Lemma \ref{lmIntersection}. For the nodes from the previous step, note that by Lemma \ref{lmIntersection} we are decreasing the size of the exceptional curve by an amount $t_k$. Let $p_k$ be the polygon formed by monodromy invariant segments through the nodes inside $F_k$. Then $p_k$ has affine distance $t_k$ from $\partial F_k$ and the rest of the proof proceeds as in the base case.
\end{proof}

\subsection{Proof of Theorem \ref{thmToricKulikov}}
\begin{proof}[Proof  of Theorem \ref{thmToricKulikov}]
    Start with a symplectic form $\omega_0$ on $X$ inducing a given form on the fiber $V$.  Pick any smooth anti-canonical divisor $H$ which intersects the toric divisors transversely. Construct $E_{sb}$ as in the discussion preceding Lemma  \ref{lmIntersection}. By Lemma \ref{lmArbitraryToricClass} we can find a symplectic form $\omega$ on $E_{sb}$ taming $J$ and whose restriction to $H$ is cohomologous to that of $\omega_0$. Two cohomologous forms which tame the same $J$ are symplectomorphic. By Corollary \ref{CYToriKulikov} $(E_{sb},\omega)$ is a symplectic Kulikov model. 

    The given hypersurface $V$ can be connected to $H$ by a path of smooth anti-canonical hypersurfaces. Thus $(V,\omega_0|_V)$ is symplectomorphic to $(H,\omega_0|_H)$ by symplectic parallel transport. The latter is symplectomorphic to the fiber $(H,\omega|_H)$ by the argument above. Thus $V$ admits a Kulikov model by definition. 

    The claim concerning the integral affine structure is the content of Proposition \ref{prpGSIAStr}.
\end{proof}

\bibliographystyle{amsalpha}
\bibliography{bibliography}
\end{document}